\documentclass[a4paper,11pt]{article}
\usepackage{amssymb}
\usepackage{amsfonts}
\usepackage{amsmath}
\usepackage{amsthm}

\usepackage[colorlinks=true]{hyperref}

\def\i{\sqrt{-1}}

\def\ddbar{\sqrt{-1}\partial\bar\partial}
\def\tr{\mathrm{tr}}

\newcommand{\Vol}{\operatorname{Vol}}

\newtheorem{remark}{Remark}[section]
\newtheorem{lemma}{Lemma}[section]
\newtheorem{definition}{Definition}[section]

\newtheorem{proposition}{Proposition}[section]
\newtheorem{corollary}{Corollary}[section]

\newtheorem{theorem}{Theorem}
\newtheorem{conjecture}{Conjecture}
\newtheorem{theorem-others}{Theorem}[section]

\begin{document}
\title{About the Calabi problem: a finite dimensional approach}
\author{H. -D. Cao and J. Keller}
\maketitle
\begin{abstract}
{Let us consider a projective manifold $M^n$ and a smooth volume form $\Omega$ on $M$. We define the gradient flow associated to the problem of $\Omega$-balanced metrics in the quantum formalism, the $\Omega$-balancing flow. At the limit of the quantization, we prove that (see Theorem 1) the $\Omega$-balancing flow converges towards  a natural flow in K\"ahler geometry, the $\Omega$-K\"ahler flow. We also prove the long time existence of the $\Omega$-K\"ahler flow and its convergence towards Yau's solution to the Calabi conjecture of prescribing the 
volume form in a given K\"ahler class (see Theorem 2). We derive some natural geometric consequences of our study.}
\end{abstract}
\section{The $\Omega$-balancing flow\label{sect1}}
In this first section we give some definitions and recall some natural moment map considerations about $\Omega$-balanced metrics. Then we introduce the two main flows of this paper,
 the $\Omega$-balancing flow and  the $\Omega$-K\"ahler flow, and state our main 
results, Theorem 1 and Theorem 2.

Assume that $M$ is a smooth polarized manifold of complex dimension $n$ and $L$ an ample line bundle. We consider  a smooth volume form $\Omega$ on $M$ such that $\int_M \Omega = \Vol_L(M):=c_1(L)^n$, the volume of $M$ with respect to $L$.

 In \cite{D3}, S.K. Donaldson introduced a notion of $\Omega$-balanced metric, adapted to the Calabi problem of fixing the volume of a K\"ahler metric in a given K\"ahler class. These metrics are algebraic metrics coming from Kodaira's embedding of the manifold in $\mathbb{P}H^0(L^k)^\vee$ for $k$ sufficiently large. More precisely, given a (smooth) hermitian metric $h\in Met(L^k)$, one can consider the 
Hilbertian map 
$$Hilb_{\Omega}=Hilb_{k,\Omega} : Met(L^k) \rightarrow Met(H^0(L^k)) $$
such that $$Hilb_{\Omega}(h)=\int_M h(.,.)\, \Omega$$ is the $L^2$ metric induced by the fibrewise $h$  and the volume form $\Omega$. 
On another hand, one can consider the Fubini-Study applications $$FS=FS_k :  Met(H^0(L^k)) \rightarrow Met(L^k) $$
such that for $H\in Met(H^0(L^ k) )$, $\{S_i\}$ an $H$-orthonormal basis of $H^0(L^k)$ and for all $p\in M$,
$$\sum_{i=1}^{\dim H^0(L^k) }|S_i(p)|^2_{FS(H)}= \frac{\dim H^0(L^k)}{\Vol_L(M)} ,$$
thus fixing pointwise the metric $FS(H)\in Met(L^k)$. One of the main result of 
\cite{D3} asserts that the dynamical system $$T_k= FS\circ Hilb_{\Omega}$$ has a unique attractive fixed point.
\begin{definition}\label{def-prop}
Let $(M,L)$ be a polarized manifold, $\Omega$ a smooth volume form. Then for any sufficiently large $k$,
there exists a unique fixed point $h_k$ of the map $T_k: Met(L^k) \rightarrow Met(L^k)$ which is called $\Omega$-balanced.
The metric $Hilb_{\Omega}(h_k)\in Met( H^0(L^k))$ and the K\"ahler form $c_1(h_k)\in 2\pi c_1(L)$, given by the curvature of $h_k$, will also be called $\Omega$-balanced.
\end{definition}

 When $k$ tends to infinity, one obtains from \cite{D3} and \cite[Theorem 3]{Ke1},  the following result.
 \begin{theorem-others}\label{DKW}
 When $k\rightarrow \infty$, the sequence of normalized $\Omega$-balanced metrics $(h_k)^{1/k}\in Met(L)$ converges
 to a hermitian metric $h_\infty$ in smooth topology and its curvature is a solution to the Calabi problem of prescribing the volume form\footnote{Note that in the rest of the paper we shall forget the normalization factor $\frac{1}{n!}$ in front of the Monge-Amp\`ere mass $c_1(h_\infty)^n$.} in a given K\"ahler class, $$c_1(h_\infty)^n = \Omega.$$ 
 \end{theorem-others}

Let us denote in the sequel $N+1=N_k+1=\dim H^0(L^k)$. Another way of presenting the notion of $\Omega$-balanced metric is to introduce a moment map setting. Let us consider first 
$\mu : \mathbb{CP}^{N} \rightarrow i \mathfrak{u}(N+1)$ which is a moment map for the $U(N+1)$ action and the Fubini-Study metric $\omega_{FS}$ on $\mathbb{CP}^{N}$. Note that here we identify
implicitly the Lie algebra $\mathfrak{u}(N+1)$ with its dual using the bilinear form $(A,B)=-\tr(AB)$. Given homogeneous unitary coordinates, one sets explicitly  $\mu=(\mu)_{\alpha,\beta}$ as 
\begin{equation}
\left(\mu([z_0,...,z_N])\right)_{\alpha,\beta}=\frac{z_\alpha z_\beta}{\sum_i |z_i|^2}. \label{mu}
\end{equation}
Then,  given an holomorphic embedding $\iota:M\hookrightarrow  \mathbb{P}H^0(L^k)^\vee$, we can consider the integral of $\mu$ over  $M$ with respect to the volume form: 
$$\mu_\Omega(\iota)=\int_M \mu(\iota(p)) \Omega(p) $$
which induces a moment map for the $U(N+1)$ action over the space of all bases of $H^0(L^k)$. Let us give some details on that point. On the space $\mathfrak{M}$ of smooth maps from $M$ to $ \mathbb{P}H^0(L^k)^\vee$, we have a natural symplectic structure $\varpi$ defined by
$$\varpi(a,b)=\int_M (a,b) \Omega.$$
for $a,b\in T_{\iota}\mathfrak{M}$ and $(.,.)$ the Fubini-Study inner product induced on the tangent vectors.
Let $\zeta \in \mathfrak{u}(N+1)$ and $X_\zeta\in H^0( {\mathbb{P}^N}^\vee, {T \mathbb{P}^N}^\vee)$ be the induced holomorphic vector field on $  {\mathbb{P}^N}^\vee=\mathbb{P}H^0(L^k)^\vee$. For all $Y\in \Gamma(M,T{\mathbb{P}^N}^\vee_{|M})$ we have that 
\begin{eqnarray*}
\varpi({X_\zeta}_{|M},Y)&=&\int_M i_Y (i_{X_\zeta} \omega_{FS}) \Omega \\
&=& -\int_M \tr(d\mu(Y)\cdot \zeta)\Omega\\
&=& -\tr(  d\mu_{\Omega}(Y) \cdot \zeta)\\
&=&(  d\mu_{\Omega}(Y) , \zeta),
\end{eqnarray*}
and $\mu_{\Omega}$ is $Ad$-equivariant as the integral of the $Ad$-equivariant moment map $\mu$.
Thus, $U(N+1)$ acts isometrically on $\mathfrak{M}$ with the moment map given by 
$$\iota\mapsto-\sqrt{-1}\left(\mu_\Omega(\iota) - \frac{\tr(\mu_\Omega(\iota))}{N+1}Id_{N+1}\right)\in i\mathfrak{su}(N+1).$$
Note that if one defines a hermitian metric $H$ on $H^0(L^k)$, one can consider an orthonormal basis with respect to $H$ and the associated embedding,  and thus it also makes sense to speak of $\mu_\Omega(H)$. As we shall see, in the Bergman space $\mathcal{B}=\mathcal{B}_k=GL(N+1)/U(N+1)$, we have a preferred metric associated to the volume form $\Omega$ and the moment map we have just defined, and this is precisely an $\Omega$-balanced metric.

\begin{definition}
The embedding $\iota$ is $\Omega$-balanced if and only if $$\mu^0_\Omega(\iota):=\mu_\Omega(\iota) - \frac{\tr(\mu_\Omega(\iota))}{N+1}Id_{N+1}=0.$$
\end{definition}
An $\Omega$-balanced  embedding corresponds (up to $SU(N+1)$-isomorphisms) to an $\Omega$-balanced metric $\iota^*\omega_{FS}$ by pull-back of the Fubini-Study metric from $\mathbb{P}H^0(L^k)^\vee$, so our two definitions actually coincide (see \cite{D3}). Note that for $H\in Met(H^0(L^k))$,  it also makes sense to consider $\mu_\Omega(h)$ where $h=FS(H)\in Met(L^k)$, i.e when $h$ belongs to the space of  \textit{Bergman} type fibrewise metric that we identify with $\mathcal{B}$.

\bigskip 

\par On the other hand, seen as a hermitian matrix, $\mu^0_\Omega(\iota)$ induces a vector field on $\mathbb{CP}^{N}$. Thus, inspired from \cite{Fi1}, we study the following flow
\begin{equation*}
\frac{d \iota(t)}{dt} =  -  \mu^0_\Omega(\iota(t)),  
\end{equation*}
and we call this flow the $\Omega$-balancing flow. To fix the starting point of this flow, we choose a K\"ahler metric
$\omega=\omega(0)$ and we construct a sequence of hermitian metrics $h_k(0)$ such that $\omega_k(0):=c_1(h_k(0))$ converges
smoothly to $\omega(0)$ providing a sequence of embeddings $\iota_k(0)$ for $k>>0$. Such a sequence of embeddings is known to exist thanks to Theorem \ref{tian}. For technical reasons, we decide to rescale this flow by 
considering the following ODE.
\begin{equation}
\frac{d \iota_k(t)}{dt} =  - k \mu^0_\Omega(\iota_k(t))  \label{resbalflow}
\end{equation}
that we call the rescaled $\Omega$-balancing flow. Of course, we are interested in the behavior of the sequence of K\"ahler metrics 
$\omega_{k}(t)=\frac{1}{k}\iota_k(t)^*(\omega_{FS})$ when $t$ and $k$ tends to infinity. Here is one of the main results of this paper.

\begin{theorem}\label{mainthm}
For any fixed $t$, the sequence $\omega_{k}(t)$ converges in $C^\infty$ topology to the solution $\omega+ \ddbar \phi_t$ of the following Monge-Amp\`ere equation 
\begin{equation}
\frac{\partial\phi_t}{\partial t} = 1-\frac{\Omega}{(\omega+ \ddbar \phi_t)^n}\label{MA}
\end{equation}
with $\phi_0=0$ and $\omega= \lim_{k\rightarrow \infty}\omega_{k}(0)$. Furthermore, the convergence is $C^1$ in the variable $t$. 
\end{theorem}
We call the flow given by Equation (\ref{MA}), the  $\Omega$-K\"ahler flow. 
The proof of this theorem will be done in several steps. First we study the limit of a convergent sequence of rescaled $\Omega$-balancing flows to identify the limit (Section \ref{study}), that we shall call the $\Omega$-K\"ahler flow.  Then, in Section \ref{Kahlerflow}, we study in details the behavior of the $\Omega$-K\"ahler flow in any K\"ahler class and prove our second main result. 

\begin{theorem}\label{mainthm2}
Let $\phi_t$ be the solution to Eq. (\ref{MA})
on the maximal time interval $0\leq t < T_{max}$. Let $$v_t=\phi_t - \frac{1}{\Vol_L(M)}\int_M \phi_t \frac{\omega^n}{n!}.$$ Then the $C^\infty$ norm of $v_t$ are uniformly bounded for all $0\leq t < T_{max}$ and consequently $T_{max}=+\infty$. 
Moreover, $v_t$ converges when $t\rightarrow \infty$ to $v_\infty$ in smooth topology and $\frac{\partial \phi_t }{\partial t}$ converges to a constant in smooth topology. 
\end{theorem}

Finally, inspired from the work of \cite{D1} and especially \cite{Fi1} for the Calabi flow, we will prove Theorem \ref{mainthm} in Section  \ref{proof}. In Section 5, we give a moment map interpretation of the $\Omega$-K\"ahler flow  and draw some possible generalizations.

\bigskip
\vspace{0.2cm}
\noindent \thanks{{\bf Acknowledgments}. \ The research of the first author is partially supported by NSF grant DMS-0909581. He would like to thank Xiaofeng Sun for helpful conversations. The second author is very grateful to J. Fine for enlightening conversations about Section \ref{sympl}. He also wishes to thank deeply R. Berman, S. Boucksom, Z. Blocki, V. Guedj and P. Eyssidieux. Both authors are grateful to the referee whose insightful comments improved a lot the preliminary version.}

\section{The limit of the rescaled $\Omega$-balancing flow}\label{study}
In this section, we assume that the sequence $\omega_{k}(t) $ is convergent and we want to relate its limit to Equation (\ref{MA}). The goal of this section is to prove the following result.
\begin{theorem}\label{identify}
Suppose that for each $t\in \mathbb{R}_+$, the metric $\omega_k(t)$ induced by Equation (\ref{resbalflow}) converges in
smooth topology to a metric $\omega_t$ and that this convergence is $C^1$ in $t\in \mathbb{R}_+$. Then
the limit $\omega_t$ is a solution to the flow (\ref{MA}) starting at $\omega_0=\lim_{k\rightarrow \infty} \omega_k(0)$.
\end{theorem}

\medskip 

\par  Given a matrix $H$ in $Met(H^0(L^k))$, we obtain a vector field $X_H$ which induces a perturbation of any embedding $\iota: M\hookrightarrow \mathbb{P}H^0(L^k)^\vee$. The induced
infinitesimal change in $\iota^* \omega_{FS}$ is pointwisely given by the potential $\tr(H\mu)$ where $\mu$ is given by (\ref{mu}). Thus, the corresponding potential in the case of the rescaled $\Omega$-balancing flow is $$ -k \tr(\mu^0_\Omega \mu).$$
Since we are rescaling the flow in (\ref{resbalflow}) and considering forms in the class $2\pi c_1(L)$, we are lead to understand the asymptotic behavior when $k\rightarrow \infty$ of the potentials 
\begin{equation} \label{potential}
\beta_k=- \tr(\mu^0_\Omega \mu) 
\end{equation}
We will need the following key result. Let us fix a K\"ahler form $\omega\in 2\pi c_1(L)$ and write $\omega= c_1(h)$. Let us define a different Hilbertian map by considering $${Hilb}: Met(L^k) \rightarrow Met(H^0(L^k))$$ by fixing 
$${Hilb}(h)=\int_M h(.,.) \omega^n=\int_M h(.,.) c_1(h)^n.$$
\begin{theorem-others}[Asymptotic expansion of the Bergman kernel]\label{tian}\- \\
The Bergman function associated to $h^k$ has the following pointwise asymptotic expansion: 
$$ \rho_k(h)(p):=\sum_{i=1}^{N+1} |s_i|^2_{h^k}(p)= k^n + \sum_{i\geq 1}k^{n-i} a_i(h),$$
where $\{s_i\}_{i=1,..,N+1}$ is an $Hilb(h^k)$-orthonormal basis of $H^0(L^k)$. By $a_i(h)$ we mean terms depending on the curvature and its covariant derivatives that are uniformly bounded on $M$. If
$h$ is varying in a compact set (in smooth topology) in the space of hermitian metrics with positive curvature, then 
$$ \Big\Vert \rho_k(h)- k^n + \sum_{i= 1}^m k^{n-i} a_i(h)\Big\Vert_{C^r} \leq \frac{C}{k^{m+1}}, $$
where $C$ is uniform and only depends on $r$. \\
A direct consequence is the convergence of the sequence of Bergman metrics $\frac{1}{k}c_1(FS({Hilb}(h^k)))$ to $\omega$ in smooth topology, i.e for all $r\geq 0$, i.e.
$$\Big\Vert \frac{1}{k}c_1(FS({Hilb}(h^k))) -\omega \Big\Vert_{C^r} = O\left(\frac{1}{k}\right).$$
\end{theorem-others}
Theorem \ref{tian}  is usually called nowadays the Tian-Yau-Zelditch expansion. S. T. Yau conjectured the convergence of the Bergman metrics in \cite[Section 6.1]{Yau2}, while G. Tian proved it in \cite{Ti1} for  $C^2$ topology (and Y-D. Ruan for $C^\infty$, see \cite{Ru}) and identified $a_0=1$. The existence of the asymptotic expansion was proved S. Zelditch \cite{Ze} (and independently by D. Catlin \cite{Ca}) using Boutet de Monvel-Sj\"ostrand techniques. The uniformity of the coefficients $a_i$ appeared in \cite{Lu} and will be crucial in the rest of the paper. 
We refer to \cite{M-M} as a general survey on this topic and provides a historical perspective. 

\begin{remark}
The function $ \rho_k(h)$ is the restriction over the diagonal of the kernel of the orthogonal $L^2$-projection
(with respect to $h$ and $Hilb(h^ k)$) from the space of smooth sections of $L^k$  to the subspace of holomorphic sections. It is usually referred as the \lq\lq Bergman function\rq\rq. Note that away from the diagonal, the kernel  $\sum_{i=1}^{N+1}\langle s_i(p_1),.\rangle_{h ^k} s_i(p_2)$ vanishes asymptotically, so the geometric information is carried only by $\rho_k(h)$.
\end{remark}
In another words, the holomorphic embedding $\iota_k$ induced by the metric $Hilb(h)$ gives a sequence of metrics by pull-back of the Fubini-Study metric $\iota_k^*(\omega_{FS})$, and this sequence is convergent to the initial metric $\omega$ when $k\rightarrow \infty$.
We will also use in the rest of the paper the following technical result that can be proved with similar arguments to Theorem \ref{tian}. See 
\cite{Ze, Ca} and \cite{Bou} where is identified the first term of the asymptotic expansion.
\begin{proposition}\label{tian-b}
Let $(M,L)$ be a projective polarized manifold. Let $h\in Met(L)$ be a metric such that its curvature $c_1(h)=\omega >0$ is a K\"ahler form. Assume $\Omega>0$ to be a volume form with continuous density. Then we have the following asymptotic expansion for $k\rightarrow \infty$, 
\begin{equation}
\sum_{i=1}^{N+1}|s_i|^2 _{h^k} = k^n \frac{\omega^n}{\Omega} +O(k^{n-1}), \label{eqn12}
 \end{equation}
where $(s_i)$ is an orthonormal basis with respect to the $L^2$ inner product $\int_M h^k(.,.) \Omega=Hilb_{\Omega}(h^k)$. Here by $O(k^{n-1})$, we mean that for $r\geq 0$
$$\Big\Vert \sum_{i=1}^{N+1}|s_i|^2 _{h^k} - k^n \frac{\omega^n}{\Omega}\Big\Vert_{C^r}\leq c_rk^{n-1}$$
where $c_r$ remains bounded if $h$ varies in a compact set (in smooth topology) in the space of hermitian metrics with positive curvature.
\end{proposition}
We will also need the following important technical result, see \cite[Theorem 1]{L-M}, \cite[Theorem 7 \& 8]{Fi1}, \cite[Section 6]{M-M2}.
\begin{theorem-others}\label{quantlaplacien}
Let us consider $h\in Met(L)$ with positive curvature and the operator  on $C^\infty(M)$ given by
\begin{equation*}
Q_k(f)(p)= \frac{1}{k^n}\int_M \sum_{a,b}{\langle s_a, s_b\rangle_{h^k} (q)\langle s_a, s_b\rangle_{h^k}(p)}f(q)\Omega(q).
\end{equation*}
which approximates the operator $\frac{\omega^n}{\Omega}\mathrm{exp}(-\frac{\Delta}{4\pi k})$ in the following sense. For any integer $r>0$, there exists $C>0$ such that for all $k>>0$  and any function $f\in C^\infty(X)$, one has
\begin{eqnarray}
\Big\Vert \left(\frac{\Delta}{k}\right)^r \left(Q_k(f)-\frac{\omega^n}{\Omega}\mathrm{exp}\left(-\frac{\Delta}{4\pi k}\right)f\right)\Big\Vert_{L ^2}\leq \frac{C}{k}\Vert f \Vert_{L^2} \label{ineq11} \\
\Vert Q_k(f)-\frac{\omega^n}{\Omega}f\Vert_{C^r} \leq \frac{C}{k} \Vert f \Vert_{C^{r+2}} \label{ineq12}
\end{eqnarray}
where the norms are taken with respect to the induced K\"ahler form obtained from the  fibrewise metric on the polarisation $L$ and $\Delta$ is the Laplace operator for the induced K\"ahler metric. The estimate is uniform when the metric varies in a compact set of smooth hermitian metrics with positive curvature.
\end{theorem-others}

We have the following first consequence. 
\begin{proposition}\label{conv0}
Let $h_k\in Met(L^k)$  be a sequence of metrics such that ${\omega}_k:=\frac{1}{k}c_1(h_k)$ is convergent in smooth topology to the K\"ahler form $\omega$. Then the potentials $\beta_k=-\tr (\mu^0_\Omega\mu)$ (induced by the embeddings given by $Hilb_{\Omega}(h_k)$) converge in smooth topology  to the potential 
$$1-\frac{\Omega}{\omega^n}.$$ 
\end{proposition}
Note that given a form $\omega$, a sequence of Bergman metrics $h_k$ is known to exist by the previous theorem. 
\begin{proof}
Let $(s_i)$ be an orthonormal basis of $H^0(L^k)$ with respect to the metric $H_k:=Hilb_{\Omega}(h_k)$. By the discussion at the beginning of Section \ref{study}, the balancing potential at  $p\in M$ for the rescaled balancing flow is
$$\beta_k(H_k)=- \int_M \sum_{a,b} \left(\frac{\langle s_a, s_b\rangle(q)}{\sum_{i=1}^{N+1}|s_i(q)|^2} -\frac{\delta_{ab}}{N+1}\right)  \frac{\langle s_a, s_b\rangle(p)}{\sum_{i=1}^{N+1}|s_i(p)|^2}\Omega(q), $$
where $\langle .,.\rangle$ stands for the fibrewise metric $h_k$. 
By the Riemann-Roch theorem, $N+1= k^n \Vol_L(M) +O(k^{n-1})$. From Proposition \ref{tian-b}, the fact that ${\omega}_k$
is convergent to $\omega$ and the uniformity of the estimates, we obtain 
\begin{eqnarray*}
\beta_k(H_k)\hspace{-0.34cm} &=& \hspace{-0.32cm} 1\hspace{-0.04cm}-\hspace{-0.04cm}\frac{k^n}{\sum_{i=1}^{N+1}|s_i(p)|^2}\int_M \sum_{a,b}\hspace{-0.08cm}\frac{\langle s_a, s_b\rangle(q)\langle s_a, s_b\rangle(p)}{k^n}\hspace{-0.02cm}\left(\hspace{-0.04cm}\frac{1}{\frac{\omega^n}{\Omega}(q)\hspace{-0.08cm}+\hspace{-0.08cm}O(\frac{1}{k})}\hspace{-0.06cm} \right)\hspace{-0.08cm} \Omega(q) \\
&=& \hspace{-0.32cm} 1\hspace{-0.04cm}-\hspace{-0.04cm} \frac{\Omega}{\omega^n}\hspace{-0.1cm}\int_M \hspace{-0.08cm}\frac{\langle s_a, s_b\rangle(q)\langle s_a, s_b\rangle(p)}{k^n}\hspace{-0.08cm}\left(\hspace{-0.08cm}\left(\hspace{-0.08cm}1\hspace{-0.08cm}+\hspace{-0.08cm}O\hspace{-0.08cm}\left(\hspace{-0.03cm}\frac{1}{k}\hspace{-0.03cm}\right)\hspace{-0.08cm}\right)\hspace{-0.08cm}\frac{\Omega}{{\omega}^n}(q)\right)\hspace{-0.08cm}\Omega(q)\hspace{-0.04cm}+O\left(\hspace{-0.03cm}\frac{1}{k}\hspace{-0.03cm}\right).
\end{eqnarray*}
But now, from Theorem \ref{quantlaplacien}, one knows the asymptotic behavior of the quantification operator $Q_k(f)(p)= \frac{1}{k^n}\int_M \sum_{a,b}{\langle s_a, s_b\rangle(q)\langle s_b, s_a\rangle(p)}f(q)\Omega(q).$ \\
Then, for $k\rightarrow \infty$, from Inequality (\ref{ineq12}) and the uniformity of the constants, one obtains
 $$\beta_k(H_k)(p) =  1 -  \frac{\Omega}{\omega^n}Q_k\left(\frac{\Omega}{\omega}+O\left(\frac{1}{k}\right)\right).$$
The convergence of $Q_k\left(\frac{\Omega}{\omega}+O\left(\frac{1}{k}\right)\right)$ to $1+O(1/k)$ follows from the same arguments as in \cite[Pages 10-11]{Fi1} and is a consequence of (\ref{ineq11}). This gives finally the expected result.
\end{proof}

Independent of the considered flows, we have also a general result that complements Theorem \ref{tian}.
\begin{proposition}\label{c-1-conv}
Let $h(t)\in Met(L)$ be a path of hermitian metrics on $L$ with $c_1(h(t))>0$. Let us consider $h_k(t)=FS(Hilb_{\Omega}(h(t)^k))^{1/k}$, the path of induced Bergman metrics.
Then $\frac{\partial h_k(t)}{\partial t}$ converges to $\frac{\partial h(t)}{\partial t}$ as $k\rightarrow +\infty$ in $C^\infty$ topology. This convergence is uniform if $h(t)$ belongs
to a compact set in the space of positively curved hermitian metrics on $L$.
\end{proposition} 
\begin{proof}
 The proof is essentially given in a discussion at the end of \cite[Section 1.4.1]{Fi1} for the sequence $FS(Hilb(h(t)^k))^{1/k}$. Let us assume that $h(t)=h_0 e^{\phi_t}$ and that $\dot{\phi}e^{\phi_t}h_0$ is the infinitesimal change of the fibrewise metric, say at $t=0$. An infinitesimal change 
of the $L^2$ inner product corresponds to the hermitian matrix in the tangent space of the Bergman metrics 
$$A=\int_M k\dot{\phi}\langle s_a, s_b\rangle \Omega,$$
and thus the potential associated to that infinitesimal change is, after rescaling to $Met(L)$,
$$\frac{1}{k}\tr(A\mu)=\frac{1}{k}\int_M  k\dot{\phi} \sum_{a,b} \langle s_a, s_b\rangle(p) \frac{\langle s_a, s_b\rangle(q)}{\sum_{i=1}^{N+1}|s_i(p)|^2}\Omega(q),$$
where the $\{s_i\}_{i=1,..,N+1}$ form an orthonormal basis of holomorphic sections with respect to $Hilb_{\Omega}(h_0^k)$ and $h_0^k=\langle .,.\rangle$. Thus, using Proposition \ref{tian-b}, one obtains  that
\begin{eqnarray*}
 \frac{1}{k}\tr(A\mu)(p) &= & \frac{ \int_M \dot{\phi}(q) \sum_{a,b} \langle s_a, s_b\rangle(p)\langle s_a, s_b\rangle(q){\Omega(q)}} {k^{n}\left(\frac{\omega_0^n}{\Omega}(p)+O\left(\frac{1}{k}\right)\right)}\\ 
&=&\frac{1}{\frac{\omega_0^n}{\Omega}(p)+O(1/k)}Q_k\left( \dot{\phi}\right)(p),
\end{eqnarray*}
and, as $k\rightarrow +\infty$, this converges, thanks to Theorem \ref{quantlaplacien}, towards 
$\dot{\phi}(p)$ after simplification. 
\end{proof}

\begin{remark}
 Thus we have obtained the convergence of the family $h_k(t)$ in $C^1$ topology with respect to the variable $t$.
 Note that the result cannot be improved, in the sense that, thanks to a direct computation, we don't expect
a convergence in $C^2$ topology. Let us be more precise. An infinitesimal change at order 2 of the induced $L^2$ inner product along a smooth path of the form $h_0e^{\phi_t}$ corresponds to a hermitian matrix
$$B=\int_M \left((k\dot{\phi})^2+k \ddot{\phi}\right) \langle s_a, s_b\rangle\Omega$$
On another hand, the potential associated to this infinitesimal change at $p\in M$ is given after rescaling by the formula
\begin{equation}\label{form1}\frac{1}{k}\left( \tr(B\mu)-\tr(A\mu)^2\right)(p)
 \end{equation}
Actually, if we write in an orthonomal basis the potential of the metric $FS(Hilb(h(t)^k))$,  $$\varphi(t)=\log \sum_\alpha \lambda_\alpha(t)|s_\alpha|^2$$ with $\varphi(0)=\log \sum_\alpha |s_\alpha|^2$, then
$\ddot{\varphi}(t)_{\mid t=0}=\frac{\sum_\alpha (\lambda_\alpha)''(0)|s_\alpha|^2}{\sum_\alpha |s_\alpha|^2}- \left(\frac{\sum_\alpha (\lambda_\alpha)'(0)|s_\alpha|^2}{\sum_\alpha |s_\alpha|^2}\right)^2$
which shows (\ref{form1}). In order to simplify the computations, let us assume that $h_0$ is solution of the Calabi problem, i.e. $c_1(h_0)^n =\omega_0^n =\Omega$. Now, using  this assumption, Proposition \ref{tian-b}, and \cite[Theorem 4.1.2]{M-M}, 
$$\frac{1}{k}\tr(B\mu)=\frac{1}{\left(1+\frac{1}{4\pi}\frac{scal(\omega_0)}{2k}+O(\frac{1}{k^2})\right)}{Q_k\left( k\dot{\phi}^2+ \ddot{\phi}\right)} .$$
Then we can define the operator on $C^\infty(M)$, $\tilde{Q}_k(f)=\frac{1}{1+\frac{1}{4\pi}\frac{scal(\omega_0)} {2k}} Q_k\left(f\right).$
We write 
$$\frac{1}{k}\left( \tr(B\mu)-\tr(A\mu)^2\right)=\tilde{Q}_k(\ddot{\phi})+k\left(\tilde{Q}_k(\dot{\phi}^2)-\tilde{Q}_k(\dot{\phi})^2\right)+O\left(\frac{1}{k}\right)$$
Then using Theorem \ref{quantlaplacien} and \cite[Theorem 6.1]{M-M2} which gives the asymptotic expansion of $Q_k$ at second order, $\frac{1}{k}\left( \tr(B\mu)-\tr(A\mu)^2\right)$ is equal to 
\begin{eqnarray*}&=&\ddot{\phi}+O\left(\frac{1}{k}\right)\\
&&+ \frac{1}{1+\frac{1}{4\pi}\frac{scal(\omega_0)} {2k}}  k \left( \dot{\phi}^2+\frac{1}{k}\left(\frac{scal(\omega_0)}{8\pi}\dot{\phi}^2- \frac{1}{4\pi}\Delta_{\omega_0}\dot{\phi}^2\right)+O\left(\frac{1}{k^2}\right)\right)\\
&&-  \left( \frac{1}{1+\frac{1}{4\pi}\frac{scal(\omega_0)} {2k}}\right)^2 k \left(\dot{\phi}
+\frac{1}{k}\left(\frac{scal(\omega_0)}{8\pi}\dot{\phi}- \frac{1}{4\pi}\Delta_{\omega_0}\dot{\phi}\right) \right)^2\\
&=&\ddot{\phi} - \frac{1}{4\pi}\Delta_{\omega_0}\dot{\phi}^2+2\dot{\phi}\frac{1}{4\pi}\Delta_{\omega_0}\dot{\phi}+O\left(\frac{1}{k}\right) \\
&=& \ddot{\phi}- \frac{1}{2\pi} \Vert \nabla \dot{\phi} \Vert^2,
\end{eqnarray*}
which is different from $\ddot{\phi}$.
\end{remark}

\medskip

We are now ready for the proof of Theorem \ref{identify} which identifies the limit of the sequence of rescaled $\Omega$-balancing flows for $k\rightarrow +\infty$. 

\medskip
\noindent {\sl Proof of Theorem \ref{identify}.} \ 
We write $\omega_t=\omega + \ddbar \phi_t$. Using the $C^1$ convergence in $t$,  $\dot{\phi}_t$ is continuous and unique up to a constant that we shall fix by setting $\int_M \dot{\phi_t} \omega_t^n=0$. Consider the potential $\beta_k(\iota_k(t))$ induced by the embedding $\iota_k(t)$ given by the rescaled $\Omega$-balancing flow at time $t$. Thanks to Proposition \ref{c-1-conv} and the fact that $\int_M \beta_k(\iota_k(t))\omega_k^n(t)\rightarrow 0$ when $k\rightarrow +\infty$, this sequence of potentials converges to $\dot{\phi_t}$. Moreover, using the balancing condition, we can apply
Proposition \ref{conv0} to get $$\dot{\phi_t}=\lim_{k\rightarrow \infty} \beta_k(\iota_k(t))=1-\frac{\Omega}{\omega_t^n}.$$ 
\qed

\section{The $\Omega$-K\"ahler flow and the proof of Theorem 2\label{Kahlerflow}}

\subsection{The long time existence}\label{existence}
We are now interested in the flow 
\begin{equation}
\frac{\partial\phi_t}{\partial t} = 1-\frac{\Omega}{(\omega+ \ddbar \phi_t)^n} \label{MA2}
\end{equation}
over a compact K\"ahler manifold (not necessarily in an integral K\"ahler class), where $\phi_0=0$ and
$\omega$ is a K\"ahler form in a fixed class $[\alpha]$. Of course, this can be rewritten as
\begin{equation}
 (\omega+ \ddbar \phi_t)^n =\frac{1}{1-\frac{\partial \phi_t}{\partial t}} e^f \omega^n \label{omega-flow}
\end{equation}
where $f$ is a smooth (bounded) function defined by $f=\log( \Omega/\omega^n)$. We are interested in the long time existence of this flow and its convergence. We study now long time existence and convergence of this flow, following the ideas of \cite{Cao}. Note that after we wrote this article we have been informed that similar results were proved recently in \cite{FLM} , and we would like to thank Prof. Z. Blocki for pointing out
this reference to us.  In this section we will prove the following result.
\begin{theorem-others}\label{longtime}
Let $\phi_t$ be the solution of 
\begin{equation*}
\frac{\partial\phi_t}{\partial t} = 1-\frac{\Omega}{(\omega+ \ddbar \phi_t)^n} 
\end{equation*}
on the maximal time interval $0\leq t < T_{max}$. Let $v_t=\phi_t - \frac{1}{\Vol_L(M)}\int_M \phi_t {\omega^n}$. Then the $C^\infty$ norm of $v_t$ are uniformly bounded for all $0\leq t < T_{max}$ and $T_{max}=+\infty$. 
\end{theorem-others}
We remark that if we look at the formal level of this equation in terms of cohomology class, we obtain directly
$$\frac{\partial [(\omega + \ddbar \phi_t)]}{\partial t} = 0,$$
which shows that the K\"ahler form $$\omega_t:=\omega + \ddbar \phi_t$$ remains in the same class as $\omega+\ddbar \phi_0$, i.e. $[\alpha]$. 
\begin{proposition}\label{lemmac0estimate}
The function $\frac{\partial\phi_t}{\partial t}$ and $\frac{1}{1- \frac{\partial\phi_t}{\partial t} }$ remain (uniformly) bounded in $C^0$ norm along the flow given by Equation (\ref{omega-flow}).
\end{proposition}
\begin{proof}
 Let us differentiate Equation (\ref{MA2}), we obtain
$$ \frac{\partial}{\partial t}\left( \frac{\partial \phi_t}{\partial t}\right) = \frac{\Omega}{\omega_t^n}\Delta_t \left( \frac{\partial \phi_t}{\partial t}\right) $$
with $\Delta_t$ the normalized Laplacian with respect to the metric $\omega+ \ddbar \phi_t$. We apply now the maximum
principle for parabolic equations at the point where $\frac{\partial \phi_t}{\partial t}$ attains its maximum (respectively its minimum).
Plugging this information in (\ref{MA2}), we obtain 
$$\frac{\partial{\phi_t}}{\partial t} \leq \sup_M (1- e^f)$$
and moreover 
$$\frac{\partial{\phi_t}}{\partial t} \geq \inf_M(1-e^f).$$
On another hand, $$\frac{\partial }{\partial t}\left( \frac{1}{1- \frac{\partial\phi_t}{\partial t} } \right) = \frac{(\omega+ \ddbar \phi_t)^n}{\Omega}\Delta_t \left( \frac{\partial \phi_t}{\partial t}\right),$$
and one applies again the maximum principle to obtain the proposition.
\end{proof}
We denote $\Delta$ the Laplacian with respect to the K\"ahler form $\omega$ given at $t=0$. 
\begin{lemma}\label{c^2low}
One has 
$$0 < n+\Delta \phi_t. $$
\end{lemma}
\begin{proof}
The fact that $\omega+\ddbar \phi_t$ is a K\"ahler form implies by taking the trace that $ n+\Delta \phi_t>0$. 
\end{proof}

We show now the upper bound for the Laplacian of the potential. 
\begin{proposition}\label{c^2}
 There exist positive constants $C_1$ and $C_2$ such that

$$ 0<n+\Delta \phi_t\leq C_1 e^{C_2(\phi_t-\inf_{M\times [o,T)} \phi_t)}, \quad for \ all \ t\in [0,T).$$

\end{proposition}

\begin{proof} In the proof we denote $\phi_t$ by $\phi$, omitting the subscript for the sake of clearness. Moreover, $g$ (resp $g_t$) denote the Riemannian metric associated to the K\"ahler form $\omega$ (resp. $\omega_t=\omega+\ddbar \phi_t$).
\par  First of all, using  holomorphic normal coordinates system at any point $p\in M$, we have

$$\Delta_t (n+\Delta\phi)= g^{k\bar l}_t (g^{i\bar j}\phi_{i\bar j})_{k\bar l} 
=g^{k\bar l}_t R_{i\bar j k\bar l}\phi_{j\bar i}+ g^{k\bar l}_t g^{i\bar j}\phi_{i\bar j k\bar l}. $$ 

Set  $$\hbar=\log \frac{\omega^n_t}{\Omega}=\log\omega^n_t -\log\omega^n-f. $$ so that $$e^{-\hbar}=\frac{\Omega}{\omega^n_t}.$$
The idea of the proof is essentially to apply maximum principle to the quantity $(n+\Delta \phi)$ with the operator  $e^{-\hbar}\Delta_t -\frac{\partial}{\partial t}$.
 
 Now, by using holomorphic normal coordinates and direct computations, we get 
 
 $$\Delta \hbar=-g^{i\bar q}_t g^{p\bar j}_t \phi_{i\bar j k}\phi_{p\bar q \bar k} +g^{i\bar j}_t (-R_{i\bar j} +\phi_{i\bar j k \bar k})+R-\Delta f. $$
 Here $R_{i\bar j k\bar l}$ and $R=scal(\omega)$ denote the curvature tensor and the scalar curvature of the metric $g_{i\bar j}$ respectively.  Then 

\begin{align*}
\frac{\partial}{\partial t} (n+\Delta \phi) =  & \Delta ( \frac {\partial \phi} {\partial t})=-\Delta (e^{-\hbar})
 = e^{-\hbar} (\Delta \hbar-|\nabla \hbar|^2)\\
= & e^{-\hbar} (g^{i\bar j}_t g^{k\bar l} \phi_{i\bar j k\bar l } - g^{i\bar j}_t R_{i\bar j} + R- \Delta f - g^{i\bar q}_t g^{p\bar j}_t \phi_{i\bar j k} \phi_{p\bar q \bar k} -|\nabla \hbar|^2).
\end{align*}

Thus
\begin{align*}
(e^{-\hbar}\Delta_t -\frac{\partial}{\partial t})(n+\Delta \phi) = & e^{-\hbar} [g^{k\bar l}_t g^{i\bar j}(\phi_{i\bar j k\bar l}-\phi_{k \bar l i\bar j}) +g^{k\bar l}_t R_{i\bar j k\bar l}\phi_{j\bar i} \\
& +g^{i\bar j}_t R_{i\bar j} -R+\Delta f + g^{i\bar q}_t g^{p\bar j}_t \phi_{i\bar j k} \phi_{p\bar q \bar k} +|\nabla \hbar|^2].
\end{align*}
On the other hand, by commuting the covariant derivatives, we have 
$$\phi_{i\bar j k\bar l}-\phi_{k\bar l i\bar j}=R_{i\bar q k\bar l}\phi_{q\bar j}-R_{i\bar j k\bar q}\phi_{q\bar l}. $$ 
Hence
\begin{align*}
(e^{-\hbar}\Delta_t -\frac{\partial}{\partial t})(n+\Delta \phi)= &e^{-\hbar}[2g^{k\bar l}_tR_{i\bar j k\bar l}\phi_{j\bar i}- g^{k\bar l}_t R_{k\bar q}\phi_{q\bar l}\\
& +g^{i\bar j}_t R_{i\bar j} -R+\Delta f + g^{i\bar q}_t g^{p\bar j}_t \phi_{i\bar j k} \phi_{p\bar q \bar k} +|\nabla \hbar|^2].
\end{align*}

Moreover, if we choose another coordinates system so that $g_{i \bar j}=\delta_{i\bar j}$ and $\phi_{i\bar j}=\phi_{i\bar i}\delta_{i\bar j}$,

\begin{align*}
g^{k\bar l}_t R_{i\bar j k\bar l}\phi_{j\bar i}- g^{k\bar l}_t R_{k\bar q}\phi_{q\bar l} = & \sum_{i,k}R_{i\bar i k\bar k} (\frac{\phi_{i\bar i}}{1+\phi_{k\bar k}}- \frac{\phi_{k\bar k}}{1+\phi_{k\bar k}})\\
= & \sum_{i,k}R_{i\bar i k\bar k} \frac{\phi_{i\bar i}^2-\phi_{i\bar i}\phi_{k\bar k}}{(1+\phi_{i\bar i})(1+\phi_{k\bar k})}\\
=& \frac{1}{2} \sum_{i,k}R_{i\bar i k\bar k} \frac{(\phi_{i\bar i}-\phi_{k\bar k})^2}{(1+\phi_{i\bar i})(1+\phi_{k\bar k})},
\end{align*}
and 
\begin{align*}
g^{k\bar l}_tR_{i\bar j k\bar l}\phi_{j\bar i} +g^{i\bar j}_t R_{i\bar j} -R= & \sum_{i,k}R_{i\bar i k\bar k} (\frac{\phi_{i\bar i}}{1+\phi_{k\bar k}} +\frac{1}{1+\phi_{k\bar k}}-1)\\
= & \frac{1}{2} \sum_{i,k}R_{i\bar i k\bar k} \frac{(\phi_{i\bar i}-\phi_{k\bar k})^2}{(1+\phi_{i\bar i})(1+\phi_{k\bar k})}.
\end{align*}
Therefore, 
\begin{equation} (e^{-\hbar}\Delta_t -\frac{\partial}{\partial t})(\Delta \phi)=\hspace{-0.05cm} e^{-\hbar}[\sum_{i,k} \frac{R_{i\bar i k\bar k} (\phi_{i\bar i}-\phi_{k\bar k})^2}{(1+\phi_{i\bar i})(1+\phi_{k\bar k})}
+\Delta f + g^{i\bar q}_t g^{p\bar j}_t \phi_{i\bar j k} \phi_{p\bar q \bar k} +|\nabla \hbar|^2]. \label{ineq1}
\end{equation}
Now, we assume the curvature tensor $R_{i\bar j k\bar l}$ is bounded below by $-C_{0}$, for some constant $C_{0}>0$,  so that 
$$ R_{i\bar j k\bar l}\geq -C_{0}(g_{i\bar j}g_{k\bar l}+g_{i\bar l}g_{k\bar j}).$$ Then, from (\ref{ineq1}) we obtain 
\begin{equation}
(e^{-\hbar}\Delta_t -\frac{\partial}{\partial t})(\Delta \phi)\geq e^{-\hbar}[-2C_{0}(\sum_{i,k} \frac{1+\phi_{i\bar i}} {1+\phi_{k\bar k}} -n^2) 
+\Delta f + g^{i\bar q}_t g^{p\bar j}_t \phi_{i\bar j k} \phi_{p\bar q \bar k}]. \label{ineq2}
\end{equation}

Finally, we consider the function $e^{-C\phi}(n+\Delta \phi)$ and compute

\begin{align*}
\Delta_t (e^{-C\phi}(n+\Delta \phi))= &\  C^2e^{-C\phi}(n+\Delta \phi)g^{i\bar j}_t \phi_i\phi_{\bar j}\\
& -Ce^{-C\phi}g^{i\bar j}_t [\phi_i(\Delta\phi)_{\bar j} +(\Delta\phi)_i\phi_{\bar j}]\\
& -Ce^{-C\phi}(n+\Delta \phi)\Delta_t\phi + e^{-C\phi}\Delta_t (n+\Delta \phi)\\
 \geq & -(n+\Delta\phi)^{-1} e^{-C\phi}g^{i\bar j}_t (\Delta\phi)_i(\Delta\phi)_{\bar j}\\
& -Ce^{-C\phi}(n+\Delta \phi)\Delta_t\phi + e^{-C\phi}\Delta_t (n+\Delta \phi),
\end{align*} 

$$\frac{\partial}{\partial t}(e^{-C\phi}(n+\Delta \phi))=  -Ce^{-C\phi}(n+\Delta \phi)\frac{\partial}{\partial t} \phi +e^{-C\phi}\frac{\partial}{\partial t} (n+\Delta \phi). $$ Thus, 
\begin{align*} 
(e^{-\hbar}\Delta_t -\frac{\partial}{\partial t})(e^{-C\phi}(n+\Delta \phi)) \geq & -(n+\Delta\phi)^{-1} e^{-(C\phi+\hbar)}g^{i\bar j}_t (\Delta\phi)_i(\Delta\phi)_{\bar j}\\
& + e^{-C\phi} (e^{-\hbar}\Delta_t -\frac{\partial}{\partial t})(n+\Delta\phi)\\
&  -Ce^{-C\phi} (n+\Delta \phi)(e^{-\hbar}\Delta_t -\frac{\partial}{\partial t})\phi.
\end{align*} 
Now observe that, by using $g_{i \bar j}=\delta_{i\bar j}$, $\phi_{i\bar j}=\phi_{i\bar i}\delta_{i\bar j}$ and (\ref{ineq2}), we have
\begin{align*}
-(n+\Delta\phi)^{-1} g^{i\bar j}_t (\Delta\phi)_i(\Delta\phi)_{\bar j}  & +(\Delta_t -\frac{\partial}{\partial t})(n+\Delta\phi)\\
\geq & -(n+\Delta\phi)^{-1} \sum_{i}(1+\phi_{i\bar i})^{-1} |\sum_{k}\phi_{k\bar k i}|^2  \\
& + \sum_{i,j,k}(1+\phi_{i\bar i})^{-1} (1+\phi_{k\bar k})^{-1}|\phi_{i\bar j k}|^2 +\Delta f\\
& -2C_{0}(\sum_{i,k} \frac{1+\phi_{i\bar i}} {1+\phi_{k\bar k}} -n^2) \\
\geq &  -2C_{0}(\sum_{i,k} \frac{1+\phi_{i\bar i}} {1+\phi_{k\bar k}} -n^2) + \Delta f .
\end{align*}
Therefore, by taking $C=C_0+1$, 
\begin{align}
(e^{-\hbar}\Delta_t -\frac{\partial}{\partial t})(e^{-C\phi}(n+\Delta \phi)) \geq & \ e^{-(C\phi+\hbar)}(\Delta f +n^2C_{0}) \nonumber  \\ 
&  -Ce^{-(C\phi+\hbar)}(n+\Delta\phi)(n-e^\hbar\frac{\partial\phi}{\partial t}) \nonumber\\
& +(C-C_0)e^{-(C\phi+\hbar)}(n+\Delta\phi)\sum_{i}\frac{1}{1+\phi_{i\bar i}}\nonumber \\
\geq & \ e^{-(C\phi+\hbar)}(\Delta f +n^2C_{0})\nonumber \nonumber \\ 
&  -Ce^{-(C\phi+\hbar)}(n+\Delta\phi)(n-e^\hbar\frac{\partial\phi}{\partial t}) \nonumber \\
& + e^{-(C\phi+\hbar+\frac{f}{n-1})}(1- \frac{\partial\phi}{\partial t})^{\frac{-1}{n-1}}(n+\Delta\phi)^{\frac{n}{n-1}},\label{ineq3}
\end{align}
where in the last inequality we have used the arithmetic-geometric inequality
\begin{align*}
\sum_{i} \frac{1}{1+\phi_{i\bar i}} & \geq 
(\frac{\sum_i(1+\phi_{i\bar i})}{(1+\phi_{1\bar 1})
\cdots(1+\phi_{n\bar n})})^{1/n-1}\\
& =  [e^{-f} (1- \frac{\partial\phi}{\partial t})]^{1/(n-1)} (n+\Delta\phi)^{\frac{1}{n-1}}.  
\end{align*}
Now the proposition follows from the maximum principle and Proposition \ref{lemmac0estimate}. Actually, at the point $(p,t_0)$ where $(e^{-C\phi}(n+\Delta \phi)) $ achieves its maximum, the left hand side of (\ref{ineq3}) is non positive and hence
$$(n+\Delta \phi(p,t_0)))^{\frac{n}{n-1}} \leq C'(1+(n+\Delta \phi(p,t_0)))$$
with $C'$ independent of $t$. Finally, $(n+\Delta \phi(p,t_0))\leq C_1$ which gives the result. 
\end{proof}

\bigskip 
Using the fact that we are working with plurisubharmonic potentials, we get the obvious fact: 
\begin{lemma} Let us denote $$v_t =\phi_t - \frac{1}{\Vol_L(M)}\int_M \phi_t {\omega^n}$$ where $\phi_t$ is solution to Equation (\ref{omega-flow}). Then, there exist constants $c_2,c_3$ such that
\begin{eqnarray*}
\sup_{M\times [0,T]} v_t \leq c_2, \\
 \sup_{M\times [0,T]} \int_M \vert v_t \vert {\omega^n}\leq c_3 .
 \end{eqnarray*} 
\end{lemma}

\begin{proposition} \label{c^0}
There exists a constant $c_4>0$ such that 
$$\sup_{M\times [0,T]} |v_t| \leq c_4.$$
\end{proposition}
\begin{proof}
[Sketch of the proof] We apply the Nash-Moser iteration argument. 
The only major difference with \cite[Lemma 3]{Cao} is that in \cite[Equation (1.14)]{Cao}, the right hand side is bounded
by the term $$n! \int_M  \frac{(-v_t)^{p-1} } {p-1}\left(\frac{e^f} {1-\frac{\partial \phi_t} {\partial t}}-1\right){\omega^n}.$$
But now, from Proposition \ref{lemmac0estimate}, one can give the following upper bound for this term:
$$C\int_M \frac{(-v_t)^{p-1}}{p-1} {\omega^n},$$
where $C$ is a uniform positive constant. This ensures that one can applies the Nash-Moser argument. This implies in a similar way to the computations of \cite[page 364]{Cao} the $C^0$ estimate. 
\end{proof}

With Propositions \ref{c^0} and \ref{c^2} and Lemma \ref{c^2low}, one obtains a uniform bound of the quantity $n+ \Delta \phi_t=n+\Delta v_t$. 
This implies from the Schauder estimates a first oder estimate 
$$\sup_{M\times [0,T]}  |\nabla v_t | \leq c_5(\sup_{M\times [0,T]}  |\Delta v_t| + \sup_{M\times [0,T]}  |v_t|) \leq c_5'.$$
All the second order derivatives of the potential $v_t$ are bounded. From the last inequality, one sees that in normal coordinates, the terms $1+\phi_{i\bar{i}}$ is bounded from above, while from Proposition \ref{c^0} and (\ref{MA}), the term $\prod_i (1+\phi_{i\bar{i}})$ is bounded. So finally, $1+\phi_{i\bar{i}}$  is  uniformly bounded along the time.  

 \par From Calabi's work and similarly to \cite{Yau, Cao}, it is now standard that it implies also the third order estimate.  Finally, using Schauder regularity theory \cite{GT} we have proved long time existence of the $\Omega$-K\"ahler flow. This concludes the proof of Theorem \ref{longtime}.

\subsection{The convergence}
In this section, we are interested in the convergence of the $\Omega$-K\"ahler flow.
\begin{theorem-others}\label{conv}
 Let us denote $v_t =\phi_t - \frac{1}{\Vol_L(M)}\int_M \phi_t {\omega^n}$ where $\phi_t$ is solution to Equation (\ref{omega-flow}), the $\Omega$-K\"ahler flow. Then, $v_t$ converges when $t\rightarrow \infty$ to $v_\infty$ in smooth topology and $\frac{\partial \phi_t }{\partial t}$ converges to a constant in smooth topology. 
\end{theorem-others}
Note that we also refer to \cite{FLM} for an independent proof of this result.
 To prove the convergence of the $\Omega$-K\"ahler flow, we need some results of P. Li and S.T. Yau for the positive solution of the heat equation on Riemannian compact manifolds \cite[Section 2]{L-Y}. 
This takes the following form. 

\begin{proposition}\label{LY}
Let $M$ be a compact manifold of dimension $n$. Let $\gamma_{ij}(t)$ a family of Riemannian metrics on $M$ such that
\begin{enumerate}
\item $ c_0\gamma_{ij}(0) \leq \gamma_{ij}(t) \leq c_0' \gamma_{ij}(0)$,
\item $ \vert \frac{\partial \gamma_{ij} }{\partial t}\vert(t) \leq c_1 \gamma_{ij}(0) $, 
\item for the Ricci curvature, $ R_{ij}(t)\geq -K g_{ij}(0)$,
\end{enumerate}
where $c_0,c_0',c_1,K$ are positive constants independent of $t$. If we denote $\tilde{\Delta}_t$ the Laplace operator of the metric
$\gamma_{ij}(t)$, and if $\phi(p,t)$ is a positive solution of the equation 
$$\left(\tilde{\Delta}_t - \frac{\partial}{\partial t} \right)\phi(p,t)=0$$
on $M\times [0,T)$, then one has the following Harnack type inequality for any $\alpha>1$:
$$\sup_{p\in M} \phi(p,t_1)\leq \inf_{p\in M}\phi(p,t_2)\left(\frac{t_2}{t_1}\right)^{\frac{n}{2}} \exp\left( \frac{c_3}{t_2-t_1} + c_4(t_2-t_1)  \right)$$
where $c_3$ depends on $c_0'$ and the diameter of $M$ with respect to $\gamma_{ij}(0)$, $c_4$ depends on the quantities $\alpha$, $K$, $n$, $c_0'$, $c_1$, $\sup \Vert \nabla^2 \log \phi \Vert$ and $0<t_1<t_2< T$. 
\end{proposition} 
With Theorem \ref{longtime} in our hands, we shall apply Proposition \ref{LY} with $ \gamma_{ij}(t)=\frac{\omega_t^n}{\Omega}g_{i\bar{j}}(t)$ where $g_{i\bar{j}}(t)$ is the metric associated with the K\"ahler form $\omega + \ddbar \phi_t$. Thus, $\tilde{\Delta}_t=\frac{\Omega}{(\omega + \ddbar \phi_t)^n}\Delta_t$ and the potential $\phi_t$ solution of  Equation (\ref{MA}) satisfies $$\left(\tilde{\Delta}_t - \frac{\partial}{\partial t} \right)\frac{\partial\phi_t(p)}{\partial t}=0.$$ We apply the same reasoning than in \cite[Section 2]{Cao}. This turns out to show that the quantity
$$E(t)=\int_M \left(\frac{\partial \phi_t}{\partial t}- \frac{1}{\Vol_L(M)}\int_M \frac{\partial \phi_t}{\partial t} \omega_t^n\right)^2 \omega_t^n $$
is (at least exponentially fast) decreasing to $0$. The only difference  with the computation in \cite{Cao} is that we need to show that the $\gamma_{ij}(t)$ are uniformly equivalent to  $\gamma_{ij}(0)$. But this is clear because the metrics $g_{i\bar{j}}(t)$ and $g_{i \bar{j}}(0)$ are uniformly equivalent thanks to Theorem \ref{longtime}, and the same happens for their respective volume forms. So the first eigenvalue of the Laplacian $\tilde{\Delta}_t$ is under control. 
\par  Similarly to \cite[Proposition 2.2]{Cao}, we obtain now Theorem \ref{conv}. Note that a consequence of Theorem \ref{longtime}
is the existence of a convergent sequence $v(p,t_n)$ in smooth topology (with $t_n\rightarrow \infty$ when $n\rightarrow \infty$) towards a smooth function  $v_\infty$.

\subsection{Corollaries}
A direct consequence of Theorem \ref{conv} is the convergence of the $\Omega$-K\"ahler flow to the solution of the Calabi conjecture. Actually, the limit $v_\infty$ satisfies 
$$(\omega +\ddbar v_\infty)^n=(\omega +\ddbar \phi_\infty)^n = \Omega.$$
On another words, one can prescribe the volume form in a given K\"ahler class. This was first proved by S-T. Yau in \cite{Yau} and our proof uses essentially the same type of estimates. Of course, if $M$ has trivial first Chern class, then the limit metric is a Calabi-Yau metric. 
\par We also remark that one can modify slightly Equation (\ref{MA2}) if the manifold $M$ has negative first Chern class. 
In that case, it is natural to introduce the following flow:
\begin{equation}
 (\omega+ \ddbar \phi_t)^n =\frac{1}{1-\frac{\partial \phi_t}{\partial t}} e^{f+\phi_t} \omega^n \label{neg-flow}
\end{equation}
where $\omega\in -2\pi c_1(M)>0$, and $f$ is the deviation of the Ricci curvature of $\omega$, that is $Ric(\omega)+ \omega = \ddbar f$ and $\int_M \frac{1}{1-\frac{\partial \phi_t}{\partial t}} e^{f+\phi_t} \omega^n =\Vol_{K_M}(M)$.
In that case similar computations to Section \ref{existence} will involve the operator $\Delta_t - Id$ since
by differentiating (\ref{neg-flow}), one obtains
$$\frac{\partial}{\partial t}\left(\frac{\partial \phi_t}{\partial t} \right)= \frac{e^{f+\phi_t}\omega^n}{\omega_t^n}\left(\Delta_t \left(\frac{\partial \phi_t}{\partial t}\right) -\left(\frac{\partial \phi_t}{\partial t}\right)\right).$$
The uniform bound of the term $\frac{1}{1-\frac{\partial \phi_t}{\partial t}}$ can be proved in a similar way to Section \ref{existence} (Proposition \ref{lemmac0estimate}) and by maximum principle, there is a uniform bound of the potentials $\phi_t$. Thus, one obtains the convergence of $\phi_t$ when $t\rightarrow \infty$ and $\omega+\ddbar \phi_\infty$ is a smooth K\"ahler-Einstein metric with negative curvature.

\section{Proof of Theorem \ref{mainthm}}\label{proof}
In this section $(M,L)$ is a polarized manifold and we are only considering integral K\"ahler classes. The techniques we use in this section to prove Theorem \ref{mainthm} are  inspired from the techniques of \cite{Fi1}. 
\subsection{First order approximation}\label{firstorder}
We know that from any starting point $\omega=\omega_0$, there exists a solution $$\omega_t=\omega+\ddbar \phi_t$$ to the $\Omega$-K\"ahler flow from the results of Section \ref{Kahlerflow}. We can write $\omega_t=c_1(h_t)$ where
$h_t$ is a sequence of hermitian metrics on the line bundle $L$. 
Furthermore, we can construct a natural sequence of Bergman metrics  $$\hat{h}_k(t)=FS(Hilb_{\Omega}(h_t^k))^{1/k}$$ by pulling back the Fubini-Study metric using sections which are orthonormal with respect to the inner product $$\frac{1}{k^n}\int_M h_t(.,.)^k\Omega.$$
Using  Proposition \ref{tian-b} we obtain the asymptotic behavior $$\hat{h}_k(t)=\left(\frac{k^n c_1(h_t)^n}{\Omega}+O\left(\frac{1}{k}\right)\right)^{1/k}h_t$$ for $k>>1$. Thus, the sequence  $\hat{h}_k(t)$  converges to $h_t$ as $k\to \infty$. 
\par On the other hand, the rescaled $\Omega$-balancing flow provides a sequence of metrics $\omega_k(t)=c_1(h_k(t))$ which are solutions to (\ref{resbalflow}). Note that by construction, we fix $h_k(0)=\hat{h}_k(0)$ for the starting point of the rescaled $\Omega$-balancing flow.
\medskip 
\par In this section, we wish to evaluate the distance between the two metrics $h_k(t)$ and $\hat{h}_k(t)$. Since we are dealing with algebraic metrics,  we have the (rescaled) metric on Hermitian matrices given by $$d_k(H_0,H_1)=\left(\frac{\tr \;(H_0-H_1)^2}{k^2}\right)^{1/2}$$ on $Met(H^0(L^k))$ which induces a metric on $Met(L)$,  that we denote by $\mathrm{dist}_k$. 

\begin{proposition}\label{step0}
One has $$\mathrm{dist}_k(h_k(t),\hat{h}_k(t))\leq \frac{C}{k},$$ for some constant $C>0$ independent of $k$. \label{distance}
\end{proposition}
\begin{proof}
The proof is similar to \cite[Proposition 10]{Fi1}. Let us consider $e^{\phi(t)}h_0$ a family of hermitian metrics with positive curvature,  and denote $$\omega_t=c_1(e^{\phi(t)}h_0).$$ The infinitesimal change at $t$ in the $L^2$ inner product induced by this path and the volume form $\Omega$ is given by 
$$\hat{U}_{\alpha,\beta}(t)=\frac{1}{k^n}\int_M \langle s_\alpha, s_\beta\rangle\, k\dot{\phi}(t)\;\Omega$$
for $(s_\alpha)$ an orthonormal basis of $H^0(L^k)$ with respect to the $L^2$-inner product $$\frac{1}{k^n}\int_M e^{k\phi(t)}e^{k\phi(t)} \Omega^n.$$ The formula is obtained by noticing that the variation occurs with respect to the fibrewise metric. 
 Now, if furthermore $\phi(t)$ is a solution to the $\Omega$-K\"ahler flow, this infinitesimal change  is given at $\hat{h}_k(t)$ as
\begin{equation}\label{tgt1}
\hat{U}_{\alpha,\beta}(t)=\frac{1}{k^n}\int_M \langle s_\alpha, s_\beta\rangle \left(k\left(1-\frac{\Omega}{\omega_t^n} \right)\right)\;\Omega,
\end{equation}
with $(s_\alpha)$ satisfy the same assumption as above. 
\par  On another hand, the tangent (at the same point $\hat{h}_k(t)$) to the rescaled $\Omega$-balancing flow (\ref{resbalflow}) is given by directly by the moment map $\mu^0_\Omega$, and we write the infinitesimal change of the $L^2$ metric as 
\begin{equation}\label{tgt2}
 {U}_{\alpha, \beta}(t) = k\int_M \left(\frac{\delta_{\alpha\beta}}{N+1}- \frac{\langle s_\alpha, s_\beta \rangle}{\sum_{i=1}^{N+1} |s_i|^2 }\right) \Omega,
 \end{equation}
where $s_i$ are $L^2$ orthonormal with respect to the $L^2$ inner product induced by $h(t)^k$ and $\Omega$. Again, from Proposition \ref{tian-b}, one has asymptotically $${U}_{\alpha, \beta}(t)  =\hat{U}_{\alpha,\beta}(t) + \frac{1}{k^n}\int_M \langle s_\alpha, s_\beta \rangle O(1) \;\Omega.$$ 
Here the term $O(1)$ stands implicitly for a (smooth) function which is bounded independently of the variables $t$ and $k$. 
Thus, one has $$\frac{\tr \;(\hat{U}_{\alpha, \beta}(t) - U_{\alpha,\beta}(t)) ^2}{k^2} = \big\langle \frac{1}{k}O(1) , Q_k\left( \frac{1}{k}O(1)\right) \big\rangle_{L^2}.$$
We can use Theorem \ref{quantlaplacien}, Inequality (\ref{ineq11}) to obtain that $$\frac{\tr \;(\hat{U}_{\alpha, \beta}(t) - U_{\alpha,\beta}(t)) ^2}{k^2} =O(k^{-2}).$$ This shows that $d_k(\hat{U}_{\alpha, \beta}(t), U_{\alpha,\beta}(t)) )=O(1/k).$ 
If we denote by $\tilde{h}_k(t)$ the rescaled balancing flow passing through $\hat{h}_k(t_0)$ at $t=t_0$, we have just proved that 
$\tilde{h}_k(t)$ and $\hat{h}_k(t)$ are tangent up to an error term in $O(1/k)$ at $t=t_0$. On the other hand, it is clear that
 $\tilde{h}_k(t)$ and $h_k(t)$ are close when $t\rightarrow \infty$, because they are  obtained through the gradient flow of the same moment map and this gradient flow converges to the unique $\Omega$-balanced metric (this is a consequence of \cite{D3}). Thus $\mathrm{dist} (\tilde{h}_k(t),h_k(t))=O(1/k)$. This finally proves the result.
\end{proof}

\subsection{Higher order approximations}
In this section, we improve the result of the last section by constructing a new time-dependent function $$\psi(k,t)=\phi_t+\sum_{j=1}^m \frac{1}{k^j}\eta_j(t)$$ which is obtained by deforming the solution to
the $\Omega$-K\"ahler flow and which satisfies the property to be ``as close'' as we wish to the $\Omega$-Balancing flow. We will need to compare this metric to the Bergman metric $h_k(t)$. Thus,
we introduce the Bergman metric associated to $h_0e^{\psi(k,t)}$, i.e 
$$\overline{h}_k(t)=FS(Hilb_{\Omega}(h_0^ke^{k\psi(k,t)}))^{1/k}.$$
We wish to minimize the quantity $$\mathrm{dist}_k(\overline{h}_k(t),{h}_k(t))$$
by showing an estimate of the form $\mathrm{dist}_k(\overline{h}_k(t),{h}_k(t))<C/k^{m+1}$, with $C>0$ a constant independent of $k>>0$ and $t$. This is the parameter version of \cite[Theorem 26]{D1}, and Proposition \ref{step0} shows that the result holds for $m=0$. One needs to choose inductively the functions $\eta_j$ and this is done by linearizing the Monge-Amp\`ere operator. 
\par Let us give some details for the first step of the induction, that is to find $\eta_1$. Consider the tangent to the path $\overline{h}_k(t)$, 
then similarly to (\ref{tgt1}), this tangent can be written as  
$$\overline{T}_{\alpha\beta}(t)= \frac{1}{k^n}\int_M k \langle s_\alpha,s_\beta\rangle \left( 1 -\frac{\Omega}{\omega_t^n}+\frac{\dot{\eta_1}}{k}+O(1/k)\right) \; \Omega,$$
where $\omega_t=\omega+\ddbar \phi_t$ and $(s_\alpha)$ is $L^2$ orthonormal with respect to $e^{\phi_t}h_0$ and the volume form $\Omega$. On another hand, the tangent to the rescaled balancing flow at the point $\overline{h}_k(t)$ is given, similarly to (\ref{tgt2})  by 
\begin{equation} \label{eq11}
 {T}_{\alpha\beta}(t)=\frac{1}{k^n}\int_M k\langle s_\alpha,s_\beta\rangle \left(1 -\frac{\Omega}{c_1(\overline{h}_k(t))^n} +O(1/k)\right)\;\Omega.
\end{equation}
But now, $$\frac{\Omega}{c_1(\overline{h}_k(t))^n}=\frac{\Omega}{\omega_t^n} -\frac{\Omega}{\omega_t^n}\Delta_t \left(\frac{1}{k}\eta_1\right)+O(1/k^2)$$ and we can write the error term $kO(1/k)$ from (\ref{eq11}) as $$kO(1/k)=\sum_{i\geq 0} \gamma_{1,i} k^{-i}= \gamma_{1,0} + O(1/k)$$ for $\gamma_{1,i}$ smooth functions with real values depending on the metric and obtained from the Bergman function asymptotics, so 
$${T}_{\alpha\beta}(t)=\frac{1}{k^n}\int_M \langle s_\alpha,s_\beta\rangle \left(k\left( 1 - \frac{\Omega}{\omega_t^n} +\frac{\Omega}{\omega_t^n}\Delta_t  \left(\frac{1}{k}\eta_1\right)+\frac{\gamma_{1,0}}{k}\right)+O(1/k)\right)\Omega.$$
If we wish to force $d_k(\overline{T}_{\alpha\beta}(t),{T}_{\alpha\beta}(t))$ to be of size $O(1/k^2)$, we need to find $\eta_1$ such that
\begin{equation}\label{linear}
\frac{\partial\eta_1(t)}{\partial t}-\frac{\Omega}{\omega_t^n}\Delta_t\eta_1(t) =\gamma_{1,0}  
\end{equation}
for all $t\geq 0$ and $\eta_1(0)=0$. But, by the standard parabolic theory (see, e.g., \cite[Section 3.1]{Baker} for a detailed exposition), a smooth solution $\eta_1$ to the above initial-value problem exists and is unique.
Then, one obtains $$\frac{\tr\;(\overline{T}_{\alpha\beta}(t)- {T}_{\alpha\beta}(t))^2}{k^2}=\langle O(1/k^2), Q_k (O(1/k^2)) \rangle_{L^2}$$
 and we can conclude with similar arguments to Section \ref{firstorder}: there exists a constant $C>0$ independent of $t$ such that
\begin{equation}
 \frac{\tr\;(\overline{T}_{\alpha\beta}(t)- {T}_{\alpha\beta}(t))^2}{k^2}\leq \frac{C}{k^4}. \label{estimate1}
\end{equation}
 This implies, by the same arguments as in the end of the proof of Proposition \ref{distance}, that $$\mathrm{dist}_k(h_k(t),\overline{h}_k(t))\leq \frac{C}{k^2}.$$
Now, for higher order expansions, one considers higher order asymptotics in the expressions above. The same reasoning can be applied for the construction of higher order approximation. It will involve, knowing the terms $\eta_1,..,\eta_{m}$ to find $\eta_{m+1}$, solution of a similar equation to (\ref{linear}), where the (non constant) R.H.S will depend on the functions 
$\eta_j$ ($1\leq j\leq m$) computed at previous step:
\begin{equation}\label{linear2}
\frac{\partial\eta_{m+1}(t)}{\partial t}-\frac{\Omega}{\omega_t^n}\Delta_t\eta_{m+1}(t) =\gamma_{m+1,0}(\eta_1,...,\eta_m).  
\end{equation}
Again, it is possible to solve (\ref{linear2}) by inverting the operator $\frac{\Omega}{\omega_t^n}\Delta_t -\frac{\partial}{\partial t}$.
Finally, we have obtained 
\begin{theorem} \label{highapprox}
Given  solution $\phi_t$ to the $\Omega$-K\"ahler flow (\ref{MA}) and $k>>0$, there exist  functions $\eta_1,...,\eta_m$, $m\geq 1$, such that the deformation of $\phi_t$ given by the potential
$$\psi(k,t)=\phi_t+\sum_{j=1}^m \frac{1}{k^j}\eta_j(t)$$
satisfies
$$\mathrm{dist}_k(h_k(t),\overline{h}_k(t))\leq \frac{C}{k^{m+1}}.$$
Here $\overline{h}_k(t)=FS(Hilb_{\Omega}(h_0^ke^{k\psi(k,t)}))^{1/k}\in Met(L)$ is the induced Bergman metric from the potential $\psi$, $h_k(t)\in Met(L)$ is the sequence of metric obtained by the rescaled balancing flow (\ref{resbalflow}), and $C$ is a positive constant independent of $k$ and $t$. 
\end{theorem}
\begin{proof}
The only point that we did not explain earlier is that $C$ is independent of the variable $t\in \mathbb{R}_+$. This comes from the following facts. On one hand, the expansion of the Bergman function of a family of metrics $h_t$ is uniform if the metrics $h_t$ belong to a compact subset of hermitian positive metrics in $Met(L)$, see Theorem \ref{tian}. On the other hand, we have seen that the metrics involved in the $\Omega$-K\"ahler flow are in a bounded set, since $\omega_t$ is convergent in smooth topology when $t\rightarrow +\infty$ thanks to Theorem \ref{conv}. This completes the proof of Theorem \ref{highapprox}.
\end{proof}
\medskip 

Furthermore, on can improve slightly this result by showing
that one has $C^1$ convergence in $t$.
\begin{proposition}\label{highapproxd}
 Under the same assumptions and notations of the previous theorem, one has
$$\mathrm{dist}_k\left(\frac{\partial h_k(t)}{\partial t},\frac{\partial\overline{h}_k(t)}{\partial t}\right)\leq \frac{C}{k^{m}},$$
where $C$ is a uniform constant in $k$ and $t$.
\end{proposition}
\begin{proof}
 One needs essentially to give an estimate of the quantity
$${\tr\;\left(\frac{\partial \overline{T}_{\alpha\beta}(t)}{\partial t}- \frac{\partial {T}_{\alpha\beta}(t)}{\partial t}\right)^2}.$$
Let us assume that we have fixed $\eta_1$ as in the proof of the theorem, that is $m=1$. Then, as the first step, we are lead to estimate
\begin{eqnarray}
\frac{1}{k^n}\int_M \langle s_\alpha,s_\beta\rangle \; k\left( \frac{\partial}{\partial t}\left(
\frac{\Omega}{\omega_t^n}\Delta_t (\eta_1/k) - \dot{\eta}_1/k+ \gamma_{1,0}/k+O(1/k^2)\right)\right)\Omega \label{firstt}\\
+
\frac{1}{k^n}\int_M k(k\frac{\partial \phi_t}{\partial t}) \langle s_\alpha,s_\beta\rangle \left( 
\frac{\Omega}{\omega_t^n}\Delta_t (\eta_1/k) - \dot{\eta}_1/k+ \gamma_{1,0}/k+O(1/k^2)\right)\Omega \label{sndt}
\end{eqnarray}
In  (\ref{firstt}), the term $O(1/k^2)$ stands for a smooth function $r(p,k,t)$, where $p\in M$, and is uniformly bounded over $M$ and in the variable $t$. But we know that the asymptotics of the Bergman kernel is given by polynomial expressions of the curvature and its covariant derivative. We can write $r(p,k,t)=\sum_{i\geq 2} \frac{1}{k^i}r_i(p,t)$ where $r_i(p,t)$ are smooth in $t$ and $p$ variables. Thus $\Vert r(p,k,t)\Vert_{C^\infty(M)}<\frac{C_1}{k^2}$ and $\Big\Vert \frac{\partial r(p,k,t)}{\partial t}\Big\Vert_{C^\infty(M)}<\frac{C_2}{k^2}$, where $C_1,C_2$ do not depend on $k$, $t$ and $p$. The independence in the variable $t$ is again obtained from the fact that the metric $\omega_t$ along the $\Omega$-K\"ahler flow is convergent (Theorem \ref{conv}) and the uniformity of the expansion. Moreover, since $\eta_1$ is a smooth solution of (\ref{linear}) in $t$, one gets that the term (\ref{firstt}) is uniformly bounded by $C_3/k^2$ using the same argument as in the proof of Theorem \ref{highapprox}, Inequality (\ref{estimate1}).
\par On another hand, by the same reasoning, (\ref{sndt}) is uniformly bounded by $C_4/k$, where $C_4$ is independent of $t$ and $k$. This provides the result for $m=1$. 
\par The computations for $m>1$ are completely similar. Also, higher order derivatives in $t$ could be treated in a similar way. 
\end{proof}

\subsection{$L^2$ estimates in finite dimensional set-up}
We start this section by fixing some notations and giving some definitions. Let us fix a reference metric $\omega_0\in 2\pi c_1(L)$. We denote $\tilde{\omega}_0=k\omega_0$ the induced metric in $2\pi kc_1(L)$. 
We need the notion of $R$-bounded geometry in $C^r$ \cite[Secion 3.2]{D1}. We say that another metric $\tilde{\omega}\in 2\pi kc_1(L)$ has
$R$-bounded geometry in $C^r$ if $\tilde{\omega}>\frac{1}{R}\tilde{\omega}_0$ and $\Vert \tilde{\omega}-\tilde{\omega}_0 \Vert_{C^r(\tilde{\omega}_0)}<R$. We say that a basis $(s_i)$ of $H^0(M,L^k)$ is $R$-bounded if the Fubini-Study metric induced by 
the embedding of $M$ in $\mathbb{P}H^0(L^k)^\vee$ associated to the $(s_i)$ has $R$-bounded geometry. 
\par The purpose to work with $R$-bounded metric is to avoid constants depending on $k$ in the forthcoming estimates.  
Let us fix $$H_A = \sum_{i,j} A_{ij} (s_i,s_j)=\tr(A\mu)\in C^{\infty}(M),$$
where $A=(A_{ij})$ is a Hermitian matrix, $(s_i)$ is a basis of $H^0(L^k)$, and $(.,.)$ denotes the fibrewise Fubini-Study inner-product induced 
by the basis $(s_i)$. This function corresponds to the potential obtained by an $A$-deformation of the Fubini-Study metric, i.e when one is moving the Fubini-Study metric in an $Lie(SU(N+1))$ orbit.
Moreover, we denote $\Vert A\Vert_{op}=\max \frac{|A\zeta|}{|\zeta|}$ the operator norm, given by the maximum moduli of the eigenvalues of the hermitian matrix $A$, and the Hilbert-Schmidt norm $\Vert A\Vert^2 = tr(A^2)=tr(AA^*)\geq 0$. 
We will need the following result which is very general,

\begin{proposition}[{\cite[Lemma 24]{D1},\cite[Proposition 12]{Fi1}}]\label{HA} There exists $C>0$ independent of $k$, such that for any basis $(s_i)$
 of $H^0(L^k)$ with $R$-bounded geometry in $C^r$ and any hermitian matrix $A$, 
$$\Vert H_A \Vert_{C^r} \leq C \Vert \mu_\Omega(\iota) \Vert_{op}\Vert A \Vert $$ 
where $\iota$ is the embedding induced by $(s_i)$. 
\end{proposition}
\begin{proof}
By definition, $\mu_\Omega(\iota)=\int_M  \mu(\iota) \Omega$. Given a holomorphic section $s$ of $L\rightarrow M$, one 
defines a holomorphic section $\tilde{s}$ of $\bar{L}^*\rightarrow \bar M$ (here $\bar M$ is just $M$ with the opposite complex structure) thanks to the bundle isomorphism given by the fiber metric. Then, for the hermitian matrix $A$, one can define the section $\sigma_A = \sum A_{ij} s_i \otimes \tilde{s_j}$ and compute its $L^2$ norm over $M \times \bar{M}$. This $L^2$ norm is given by $tr(A \mu_\Omega\mu_\Omega^* A^*)^{1/2}$. But one has an obvious upper bound for that term, by a standard inequality: for hermitian matrices $G$,$F$, $\tr(FGF) \leq \Vert F \Vert ^2 \Vert G \Vert_{op}$. Thus,
\begin{equation}
\Vert \sigma_A \Vert = tr(A \mu_\Omega\mu_\Omega^* A^*)^{1/2}\leq   \Vert \mu_\Omega(\iota) \Vert_{op}\Vert A \Vert.\label{inter} 
\end{equation}
 On another hand, for any holomorphic section $\sigma$ of a hermitian vector bundle $\tilde{L}\rightarrow Y$, one has the $L^2$ estimate, $\Vert \sigma \Vert_{C^r(Y')} \leq C\Vert \sigma \Vert_{L^2(Y)}$ for  a submanifold $Y'\subset Y$ and some constant $C$ that depends on $Y$. This is described in \cite[Lemma 24]{D1}. Hence, applying this result with $Y=M\times \bar M$ and $Y'=M$, together with (\ref{inter}), one obtains the expected inequality.
\end{proof}

We will need the following lemma.
\begin{lemma}\label{l5}
Let us fix $r\geq 2$. Assume that for all $t\in [0,T]$, the family of basis $\{s_i\}(t)$ of $H^0(L^k)$ have $R$-bounded geometry. Let us define by $h(t)$ the family of Bergman metrics induced by $\{s_i\}(t)$. Then
the induced family of Fubini-Study metrics $\tilde{\omega}(t)$ satisfy 
$$\Vert \tilde{\omega}(0)-\tilde{\omega}(T)\Vert_{C^{r-2}}< C \sup_{t}\Vert \mu_\Omega(\iota(t))\Vert_{op}\int_0^T \mathrm{dist}(h(s),h(0))ds,$$
and also
\begin{eqnarray*}
 \Big\Vert \frac{\partial\tilde{\omega}}{\partial t}(0)-\frac{\partial \tilde{\omega}}{\partial t} (T)\Big\Vert_{C^{r-2}}&<& C^* \sup_t \Vert \mu_\Omega(\iota(t))\Vert_{op}\int_0^T \mathrm{dist}(\frac{\partial h}{\partial s}(s),\frac{\partial h}{\partial s}(0)) ds \\
&&+ C^* \sup_t  \Vert d\mu_\Omega(\iota(t))\Vert_{op}\int_0^T \mathrm{dist}(h(s),h(0)) ds,
\end{eqnarray*}
where $C,C^*$ are uniform constants in $k$.
\end{lemma}
\begin{proof}
Thanks to \cite[Lemma 13]{Fi1}, we just need to check the second inequality. For the deformation $A(t)$ of the $L^2$ metric induced along the path from 0 to $T$, one has
\begin{eqnarray}
\Big\Vert \frac{\partial^2 \tilde{\omega}(t)}{\partial t^2}\Big\Vert_{C^{r-2}}&=&\Big\Vert \sqrt{-1}\partial\bar\partial \frac{\partial}{\partial t} H_{A(t)}\Big\Vert_{C^{r-2}}\\
&\leq &\Vert \partial \bar \partial \tr(\dot{A}(t)\mu)\Vert_{C^{r-2}} +\Vert \partial \bar \partial \tr({A}(t)\dot{\mu}(\iota_t))\Vert_{C^{r-2}} \label{est3}.
\end{eqnarray}
The first term of (\ref{est3}) can be bounded from above by  $C\Vert \mu_\Omega(\iota(t))\Vert_{op} \Vert \dot{A}(t)\Vert$  using directly Proposition \ref{HA}. For the second term, one needs to
adapt the proof of Proposition \ref{HA}, but this can be done with no major difficulty. Hence, the second term of (\ref{est3}) can be bounded from above by 
\begin{eqnarray*}
\Vert \partial \bar \partial \tr({A}(t)\dot{\mu}(\iota_t))\Vert_{C^{r-2}}&\leq &C\Big\Vert \int_M \dot{\mu}(\iota(t))\Omega\Big\Vert_{op} \Vert A(t) \Vert \\
&\leq& C' \Vert d\mu_{\Omega}(\iota(t))\Vert_{op} \Vert A(t) \Vert.
 \end{eqnarray*}

Then by integration, one obtains the expected estimate.  
\end{proof}

\begin{corollary}\label{R-bounded}
Let $\tilde{\omega}_k$ be a sequence of metrics with $R/2$-bounded geometry in $C^{r+2}$ such that the norms $\Vert \mu_{\Omega}(\tilde{\omega}_k)\Vert_{op}$ are uniformly bounded. Then, there is a constant $C>0$ independent of $k$ such that if
 $\tilde{\omega}$ has $\mathrm{dist}_k(\tilde{\omega}, \tilde{\omega}_k) <  C$, then $\tilde{\omega}$ has $R$-bounded geometry in $C^r$. 
\end{corollary}
\begin{proof}
The proof is completely similar to \cite[Lemma 14]{Fi1}.
\end{proof}

\subsection{Projective estimates}
In this subsection, we aim to control the operator norm of the moment map in terms of the Riemannian distance in the Bergman space $$\mathcal{B}=GL(N+1)/U(N+1).$$ With this result in hand, we can launch the gradient flow
of the moment map and show its convergence. 
\par We start our investigation by the following result, which is a direct consequence of Theorem \ref{tian}. 
\begin{proposition}\label{cor-tian}
 Let $h$ be a hermitian metric on $L$ with curvature $\omega=c_1(h)>0$. Consider the  
sequence $h_k=FS(Hilb(h))\in Met(L^k)$ of Bergman metrics, approximating after renormalisation $h$, thanks to Theorem \ref{tian}. Let us call $$I_{\Omega,k}=\int_M \langle s_i,s_j\rangle_{h^k} \Omega$$ for $(s_i)$ a basis of holomorphic sections of $H^0(L^k)$ with respect to $Hilb(h)$. Then, when $k\rightarrow + \infty$, $$\Vert \mu_\Omega(h_k)-I_{\Omega,k}\Vert_{op}\rightarrow 0$$
 and the convergence is uniform for $\omega$ lying in a compact subset of K\"ahler metrics in $2\pi c_1(L)$.
\end{proposition}
\begin{proof}
Firstly, the matrix $I_{\Omega,k}$ does not depend on the choice of the orthonormal basis $\{s_i\}_{i=1,..,N}$. 
 Thanks to the asymptotic expansion given by Theorem \ref{tian}, 
$$\mu_\Omega(h_k)=\int_M \langle s_i,s_j\rangle_{FS(Hilb(h))} \Omega = \int_M \langle s_i,s_j \rangle_{h ^k} (1+O(1/k))\Omega.$$
Finally, we can conclude the convergence by using \cite[Lemma 28]{D1} which ensures that for the operator norm, $$\Big{\lVert}{\int_M \langle s_i,s_j\rangle_{FS(Hilb(h))}\times O\left(\frac{1}{k}\right)  \Omega}\Big\rVert_{op} \leq \Big\vert \frac{\Omega}{\omega^n}O\left(\frac{1}{k}\right)\Big\vert_{L^\infty} .$$
The uniformity of the convergence is given by the uniformity of the expansion in the asymptotics, see Theorem \ref{tian}.
\end{proof}

Given a tangent vector $A\in T_b \mathcal{B}$ with $b\in \mathcal{B}$, we  
have a vector field $\zeta_A$ on ${\mathbb{P}^N}^\vee$ and thus on $M$, corresponding to $A$. Of course, the fact that  $\mu_\Omega$ is a moment map (see Section \ref{sect1}) gives straightforward the following fact.  
\begin{lemma}\label{l1}
 For any pair of Hermitian matrices $A,B \in T_b \mathcal{B}$, one has
$$tr(B d\mu_\Omega(A))= \int_M (\zeta_A, \zeta_B) \Omega,$$
where $(.,.)$ denotes the Fubini-Study inner product induced on the tangent vectors.
\end{lemma}

\begin{lemma}{\cite[Lemma 18]{Fi1}}
 Let $A,B\in  T_b \mathcal{B}$. Pointwisely over ${\mathbb{P}^N}^\vee$, one has
$$H_AH_B + (\zeta_A,\zeta_B)=\tr(AB\mu).$$
\end{lemma}

\begin{lemma}\label{l2}
 For any hermitian matrices $A,B \in T_b \mathcal{B}$,
$$\tr(Bd\mu_{\Omega}(A)) + \langle H_A,H_B\rangle_{L^2(M,\Omega)}=tr(AB\mu_{\Omega}).$$
\end{lemma}
\begin{proof}
 We start from the previous lemma which says that at each point of $M$,
$$H_AH_B + (\zeta_A,\zeta_B)=\tr(AB\mu).$$ Now, we integrate with respect to the
volume form $\Omega$ and apply Lemma \ref{l1}.
\end{proof}

\begin{lemma}
 For any hermitian matrix $A\in T_b \mathcal{B}$,
$$\Vert H_A \Vert^2_{L^2(\Omega)}\leq \Vert A \Vert^2 \Vert \mu_{\Omega}\Vert_{op}.$$ 
\end{lemma}
\begin{proof}
 From the last lemma, 
$$\Vert H_A \Vert^2_{L^2(\Omega)} = \tr(A^2\mu_{\Omega})-\tr(Ad\mu_{\Omega}(A)).$$ Now, by Lemma \ref{l1},  
$$\tr(Ad\mu_{\Omega}(A))=\int_M (\zeta_A,\zeta_A)\Omega\geq 0.$$ Hence,
$$\Vert H_A \Vert^2_{L^2(\Omega)}\leq \tr(A^2\mu_{\Omega})\leq \Vert A \Vert^2 \Vert \mu_{\Omega}\Vert_{op}.$$
\end{proof}

\begin{lemma}\label{l3}
 For any Hermitian matrix $A\in T_b \mathcal{B}$, 
$$\Vert d\mu_{\Omega}(A) \Vert_{op} \leq \Vert d\mu_{\Omega}(A) \Vert \leq 2 \Vert A \Vert \Vert \mu_{\Omega} \Vert_{op}.$$
\end{lemma}
\begin{proof} From Lemma \ref{l2}, one has
\begin{eqnarray*}
\Vert d\mu_{\Omega}(A) \Vert^2 &=& \tr(d\mu_{\Omega}(A)^2) \\
&=& tr(Ad\mu_{\Omega}(A)\mu_{\Omega})-\langle H_A,H_{d\mu_{\Omega}(A)}\rangle_{L^2(\Omega)}\\
&\leq& \Vert A \Vert 
\Vert d\mu_{\Omega}(A)\Vert \mu_{\Omega} \Vert_{op}-\langle H_A,H_{d\mu_{\Omega}(A)}\rangle_{L^2(\Omega)}.
\end{eqnarray*} Then we can conclude by using the fact that $$\vert \langle H_A,H_{d\mu_{\Omega}(A)}\rangle_{L^2(\Omega)}\vert  \leq \Vert H_A\Vert_{L^2(\Omega)}  \Vert H_{d\mu_{\Omega}(A)}\Vert_{L^2(\Omega)},$$
and the previous lemma.
\end{proof}

Finally, we obtain 
\begin{proposition} \label{p1}
 Let $b_0, b_1 \in \mathcal{B}$. Then,
$$\Vert \mu_{\Omega}(b_1)\Vert_{op} \leq e^{2\mathrm{dist}_k(b_0,b_1)}\Vert \mu_{\Omega}(b_0)\Vert_{op}.$$
\end{proposition}

\begin{proof}
 We know that a geodesic in the space of Bergman metrics $\mathcal{B}$ is given by a line, i.e.,  the Hermitian metric involved along the geodesic is modified by $e^{tA}$ 
and that $\mathrm{dist}(b_0,b_1)=\Vert A \Vert$. This can be rephrased by saying that if $\{s^0_i\}_{i=1,..,N+1}$ (resp. $\{s^1_i\}_{i=1,..,N+1}$) is an orthonormal basis of $H^0(L^k)$ with respect to $b_0$ (resp. $b_1$), then there exists $\sigma \in GL(N+1)$ such that $\sigma \cdot s^0 = s^1$ and without loss of generality we can assume $\sigma$ diagonal with entries $e^{\lambda_0}, .. , e^{\lambda_n}$. Then the geodesic is just induced by the family of basis $\sigma^t \cdot s^0$ for $t\in[0,1]$. Now, we can conclude our proof by using Lemma
\ref{l3} and the fact that the norm $\Vert. \Vert_{op}$ on the space of matrices is controlled from above by the Hilbert-Schmidt norm $\Vert . \Vert$.
\end{proof}

We are now ready to give the proof of Theorem \ref{mainthm}, that is to show the smooth convergence of K\"ahler metrics $\omega_k(t)$ involved in the rescaled balancing flow (\ref{resbalflow}) towards the solution $\omega_t$ to the $\Omega$-K\"ahler flow. \\

\noindent{\sl Proof of Theorem \ref{mainthm}.} \ Using Theorem \ref{highapprox}, for any $m>0$, we have obtained a sequence of K\"ahler metrics $$\omega(k;t)=c_1(h_0 e^{\psi(k,t)})$$ such that $\omega(k;t)$ converges, when $k\rightarrow +\infty$ and in smooth sense, towards the solution $\omega_t=c_1(h_0 e^{\phi_t})$ to the $\Omega$-K\"ahler flow. Moreover, one has, for $k$ large enough and with
$\overline{h}_k(t)\in \mathcal{B}$ the Bergman metric associated to $h_0 e^{\psi(k,t)}\in Met(L)$, the estimate 
\begin{equation}\label{estimate}
 \mathrm{dist}_k(h_k(t),\overline{h}_k(t))\leq \frac{C}{k^{m+1}},
\end{equation}
 where $h_k(t)$ is the metric induced by the rescaled $\Omega$-balancing flow.
Consequently, in order to get the $C^0$ convergence in $t$, all what we need to show is that 
\begin{equation}
\Vert \omega_k(t)-c_1(\overline{h}_k(t))\Vert_{C^r(\omega_t)} \rightarrow 0. \label{aim} 
\end{equation}
The idea is to consider the geodesic in the Bergman space between these two points. 
\medskip

Firstly, we will get that along the geodesic from $\overline{h}_k(t)$ to $h_k(t)$ in $\mathcal{B}$,  $\Vert \mu_{\Omega}\Vert_{op}$ is controlled uniformly if we can apply Proposition \ref{p1}. This requires to prove that $\overline{h}_k(t)$ is at a uniformly bounded distance of $h_k(t)$ and that $\Vert \mu_{\Omega}(\overline{h}_k(t))\Vert_{op}$ is bounded in $k$. But, this comes from the fact that one can choose precisely $m\geq n+1$ in Inequality (\ref{estimate}) and one can apply Proposition \ref{cor-tian}. For the latter, one needs to notice the estimate $$\Vert I_{\Omega,k} \Vert_{op}\leq \sup_M \frac{\Omega}{\omega^n}$$ from \cite[Lemma 28]{D1}.
\par Secondly, we show that the points along this geodesic have $R$-bounded geometry.  This is a consequence of Corollary \ref{R-bounded}, applied with the reference metric $\omega_t$ to the sequence $c_1(\overline{h}_k(t))$. On one side, $\Vert \mu_{\Omega}(\overline{h}_k(t))\Vert_{op}$ is under control as we have just seen. On another side, $c_1(\overline{h}_k(t))$ are convergent to $\omega_t$ in $C^{\infty}$ (hence in $C^{r+4}$ topology), thus they  have $R/2$-bounded geometry. Given $m\geq n+2$, one obtains, thanks to  Corollary \ref{R-bounded} and inequality (\ref{estimate}), that all the metrics along the geodesic from $\overline{h}_k(t)$ to $h_k(t)$ have $R$-bounded geometry in $C^{r+2}$. 
\par Thirdly, we are exactly under the conditions of Lemma \ref{l5}. It gives, by renormalizing the metrics in the K\"ahler class $2\pi c_1(L)$ and by (\ref{estimate}), that
\begin{eqnarray*}
\Vert k\omega_k(t)-kc_1(\overline{h}_k(t))\Vert_{C^r(k\omega_t)}&\leq &C \Vert \mu_{\Omega}(\overline{h}_k(t))\Vert_{op}k^{n+2}\mathrm{dist}_k(h_k(t),\overline{h}_k(t)),\\
\Vert \omega_k(t)-c_1(\overline{h}_k(t))\Vert_{C^r(\omega_t)}& \leq & C\Vert \mu_{\Omega}(\overline{h}_k(t))\Vert_{op}k^{n+2-m-1+r/2},
\end{eqnarray*}
where we have used that the geodesic path from $0$ to $1$ is just a line. Here $C>0$ is a constant that does not depend on $k$. If we choose $m>r/2+1+n$, we get the expected convergence in $C^r$ topology, i.e Inequality (\ref{aim}). Of course, this reasoning works to get the uniform $C^0$ convergence in $t$ for $t\in \mathbb{R}_+$, because all the K\"ahler metrics $\omega_t$ that we are using are uniformly equivalent (we have convergence of the $\Omega$-K\"ahler flow, Theorem \ref{conv}) and because we have uniformity of the expansion in Theorem \ref{tian} and Theorem \ref{quantlaplacien}.
\par We now prove that one has $C^1$ convergence in $t$ of the flows $\omega_k(t)$. Again, we need to show the $C^1$ convergence of $\omega_k(t)$ to $c_1(\overline{h}_k(t))$, because we already know the convergence of $c_1(\overline{h}_k(t))$ to $\omega_t$ by Proposition \ref{c-1-conv}.  We are under the conditions
 of Lemma \ref{l5} by what we have just proved above. So we have, using again that our path is a geodesic, 
\begin{eqnarray*}
\Big\Vert k\frac{\partial\omega_k(t)}{\partial t}-k\frac {\partial{c_1(\overline{h}_k(t))}} {\partial t}\Big\Vert_{C^r}&\hspace{-0.3cm}\leq & \hspace{-0.3cm} C^* \Vert \mu_{\Omega}(\overline{h}_k(t))\Vert_{op}k^{n+2}\mathrm{dist}_k\left(\hspace{-0.07cm}\frac{\partial h_k(t))}{\partial t},\frac{\partial \overline{h}_k(t)}{\partial t}\hspace{-0.07cm}\right) \\
&& \hspace{-0.3cm}+ C^*\Vert d\mu_{\Omega}(\overline{h}_k(t))\Vert_{op}k^{n+2}\mathrm{dist}_k(h_k(t),\overline{h}_k(t)).
 \end{eqnarray*}
Here the $C^r$ norm is computed with respect to $k\omega_t$. If we apply Lemma \ref{l3}, Theorem (\ref{highapprox}) 
and Proposition \ref{highapproxd}, we can bound from above the RHS of the last inequality, and get
\begin{eqnarray*}
\Big\Vert \frac{\partial\omega_k(t)}{\partial t}-\frac {\partial{c_1(\overline{h}_k(t))}} {\partial t}\Big\Vert_{C^r(\omega_t)}&\leq& C'\Vert \mu_{\Omega}(\overline{h}_k(t))\Vert_{op}k^{n+2-m-r/2}\\
&&+  C'' \Vert \mu_{\Omega}(\overline{h}_k(t))\Vert_{op}k^{n+2+r/2}k^{-m-1}k^{-m-1}\\
&\leq& C'''k^{n+2-m-r/2}.
\end{eqnarray*}
Finally, we choose $m>r/2+n+2$ to obtain $C^1$ convergence. This completes the proof of Theorem \ref{mainthm}.
\qed

\bigskip
Finally, if we apply Definition-Proposition (\ref{def-prop}) that asserts that an $\Omega$-balanced metric does always exist and is a zero of the moment map $\mu_{\Omega}^0$, Theorem \ref{mainthm},  
and the convergence of the $\Omega$-K\"ahler flow towards a solution to the Calabi problem, we obtain directly the following result.
\begin{corollary}
Under the same setting as above, the sequence of balanced metric $h_k(\infty)^{1/k}\in Met(L)$, obtained as the limit of the balancing flow at $t=+\infty$, converges in smooth topology towards $h_\infty$, a solution of the Calabi problem, $$(c_1(h_\infty))^n= \Omega.$$
\end{corollary}
Note that this is a new proof of Theorem \ref{DKW}, but which uses a priori the existence of a solution to the Calabi problem (compare with \cite{Ke1}). 

\section{The infinite dimensional setup and generalizations} 
\subsection{A symplectic approach to the Calabi problem} \label{sympl}
In this section we develop the moment map set-up on the infinite dimensional space of K\"ahler potentials related to the $\Omega$-K\"ahler flow. Let us assume that $(M,L)$ is a polarized manifold.
 Let us fix $\omega\in 2\pi c_1(L)$ and $\Omega$ a smooth volume form on $M$ with $\int_M \Omega=\Vol_L(M)$.
We introduce $\mathcal{M}$ the infinite dimensional space of integrable hermitian connexions on $L$
with K\"ahler form as curvature, with respect to a fixed complex structure. It means that we consider unitary connexions $\nabla$ on $L$ such that if $F_\nabla\in \Omega^2(M,End(L))$ is the curvature connexion, then  
$F_\nabla^{0,2} =F_\nabla^{2,0}=0$ and $F_\nabla^{1,1}$ is a positive form with respect to the complex structure on $M$. Consider the abelian gauge group $\mathcal{G}$ of maps $L\rightarrow L$  that cover the identity on $M$. 
By duality, the Lie algebra  $Lie(\mathcal{G})$ can be identified with the space of smooth functions from $M$ to $\mathbb{R}$ with zero integral, since one can identify $\mathcal{G}$ with $C^{\infty}(M,S^1)$.  The tangent space at $\mathcal{M}$  is given by the $1$-forms with values in $End(L)$. For simplicity, we assume that $M$ is simply connected and 
we fix the following symplectic form on $\mathcal{M}$ at the point $\nabla \in \mathcal{M}$, 
$$\nu_\nabla(a,b)=\int_M a \wedge b \wedge F_\nabla^{n-1}$$
which is a symplectic form invariant under the action of $\mathcal{G}$. 
Note that $\nu$ is invariant under the action of the group $\mathcal{G}$. 
\par We have a natural paring $Lie(\mathcal{G})\times Lie(\mathcal{G})^*\rightarrow \mathbb{R}$ given by
\begin{equation}
(\zeta,\theta)\mapsto \int_M \zeta\theta= \int_M \langle \zeta, \theta\rangle. \label{pair}
 \end{equation}

We are in a moment map setting. Actually, we have the following simple proposition that shows that prescribing the volume form in a K\"ahler class is related to finding the zero of a certain moment map. Note that given $\nabla\in \mathcal{M}$, we set
$A_\nabla$ is the real  connection ($S^1$ invariant) 1-form associated to $\nabla$ on the natural $S^1$-principal bundle $\pi:P\rightarrow M$ that we can associate with $L\rightarrow M$. It acts on  an element of $\zeta \in Lie(\mathcal{G})$  by decomposing in a vertical and horizontal parts and fibrewise this vertical part corresponds to a rotation which is eventually parametrized by the real function $\langle A_\nabla, \zeta\rangle$ over $M$.
\begin{proposition}
 There is a moment map $\overline{\mu}: \mathcal{M}\rightarrow Lie(\mathcal{G})^*$ associated to the action of $\mathcal{G}$ on $(\mathcal{M},\nu)$ given by
$$\overline{\mu}(\nabla)= \langle A_\nabla, .\rangle((F_\nabla)^n-\Omega).$$
\end{proposition}
\begin{proof}
 We need to check that for any $\zeta \in Lie(\mathcal{G})$ and any a vector field $V$, we have 
$$\langle d\overline{\mu}(\nabla)(V),\zeta\rangle = \nu_\nabla (V,X_\zeta),$$ 
where $X_\zeta$ is the vector field on $\mathcal{M}$ defined by the infinitesimal action of $\zeta\in Lie(\mathcal{G})$. More explicitly,   $X_\zeta$ is given by $X_\zeta = L_\zeta A_\nabla = d \langle A_\nabla ,\zeta\rangle +\iota_\zeta dA=d \langle A_\nabla ,\zeta\rangle+\iota_{\pi_*\zeta} F_\nabla=d \langle A_\nabla ,\zeta\rangle$, since the elements of $\mathcal{G}$ cover the identity on $M$. Now, we have
\begin{eqnarray*}
 \nu_\nabla (V,X_\zeta)&=& \int_M V \wedge d \langle A_\nabla ,\zeta\rangle \wedge F_\nabla^{n-1}\\
  &=& \int_M \langle A_\nabla, \zeta\rangle dV  \wedge  F_\nabla^{n-1}.
\end{eqnarray*}
But the change in $F_\nabla$ by the vector field $V$ is precisely given by $dV$, so 
$$\langle d \overline{\mu}(\nabla)(V),\zeta\rangle=\int_X \langle A_\nabla, \zeta\rangle  dV\wedge  F_\nabla^{n-1},  $$
since the elements of $\mathcal{G}$ cover the identity on $M$. 
\end{proof}
Note that the moment map that we have just defined is obviously not unique. 
Let us denote as in the proof above by $X_\zeta$ the vector field associated to $\zeta\in Lie(\mathcal{G})$. Now, using the pairing (\ref{pair})
and the natural $\mathcal{G}$-invariant norm on $Lie(\mathcal{G})$, one can consider $\overline{\mu}$ with values in $Lie(\mathcal{G})$, which means that
we write $\overline{\mu}(\nabla)=\frac{(F_\nabla^ n - \Omega)}{F_\nabla^ n}$. Then, we consider the gradient flow
$$\frac{d}{dt}\Vert \overline{\mu}(\nabla_t)\Vert^2 = -\Vert X_{\overline{\mu}(\nabla_t)}\Vert^2$$
where the norm on the R.H.S is computed with respect to $\nu_{\nabla_t}$.
This is actually equivalent to 
\begin{equation}
\frac{dA_{\nabla_t}}{dt}= IX_{\overline{\mu}(\nabla_t)}, \label{flow-moment}
\end{equation}
with $I$ the complex multiplication on the tangent vectors in $\mathcal{M}$. This equation can be rephrased in terms of flow over 1-forms by 
$$\frac{d F_{\nabla_t}}{dt}= L_{I{\overline{\mu}(f_t)}} F_{\nabla_t}$$
If we use the notations of the previous sections, $\omega_t=F_{\nabla_t}$ is an evolving K\"ahler form, then this (negative) gradient flow reads
as 
$$\frac{d \omega_t}{dt}=\i\partial\bar{\partial}\left(\frac{\omega_t^ n - \Omega}{\omega_t^ n}\right).$$
Then, using the fact that the kernel of the operator $\i\partial\bar{\partial}$ is given by constants (since $M$ is compact), one recovers precisely the equation of the $\Omega$-K\"ahler flow (\ref{MA2}).
Finally, we would like to mention that J. Fine has developed in the preprint \cite{Fi2} a more general theory that covers the results presented in this section (see \cite[Section 3.2]{Fi2}).
\subsection{Integral of a moment map}

In this section we deal with a very general setup. Consider the case of a K\"ahler manifold
 $(\Xi,\omega)$ polarized by the line bundle $\mathsf{L}$ and a moment map $\mu$ associated to the action of
a linear reductive group $\Gamma$ such that its complexified acts holomorphically.
To the moment map $\mu$  corresponds canonically a functional
\begin{equation*}
\Psi:\Xi \times \Gamma ^{\mathbb{C}}\rightarrow \mathbb{R}
\end{equation*}
that we call \lq\lq\textit{integral of the moment map} $\mu$\rq\rq \
and that satisfies the following two properties:

\begin{itemize}
\item For all $p\in \Xi,$ the critical points of the restriction $\Psi
_{p}$ of $\Psi $ to $\{p\}\times \Gamma ^{\mathbb{C}}$ 
coincide with the points of the orbit $Orb_{\Gamma ^{\mathbb{C}
}}(p)$ on which the moment map vanishes;
\item the restriction $\Psi _{p}$ on the \lq\lq lines\rq\rq \ $\{e^{\lambda u}:u\in\mathbb{R}
\}$ where $\lambda \in Lie\left( \Gamma ^{\mathbb{C}}\right) $ is convex.
\end{itemize}
This is well known in the projective case from the seminal work of G.  Kempf and L. Ness \cite{K-N}. We refer also to
\cite{MR} for a more general setting.
\begin{theorem-others}
There exists a unique map $\Psi:\Xi \times \Gamma ^{
\mathbb{C}}\rightarrow
\mathbb{R}
$ that satisfies:

\begin{enumerate}
\item $\Psi \left( p,e\right) =0$ for all $p\in \Xi;$
\item $\frac{d}{du}\Psi \left( p,e^{i\lambda u}\right) _{|u=0}=\langle\mu \left(
p\right) ,\lambda \rangle$ for all $\lambda \in Lie\left( \Gamma \right).$
\end{enumerate}
\end{theorem-others}

Let us sum up some of the main properties of the integral of the moment map.
\begin{proposition}
\noindent The functional $\Psi $ is $\Gamma-$invariant (for the left action) and satisfies the cocylicity relation
\begin{equation*} \Psi \left( p,\gamma \right) +\Psi \left(\gamma p,\gamma ^{\prime }\right) =\Psi \left( p,\gamma ^{\prime }\gamma \right)
\end{equation*}
 for all $p\in \Xi ,$ $\gamma ,\gamma' \in \Gamma ^{\mathbb{C}}$, and the relation of equivariance 
$$\Psi \left( \gamma p,\gamma ^{\prime }\right)
=\Psi \left( p,\gamma ^{-1}\gamma ^{\prime }\gamma \right)$$ 
for all $p\in
\Xi ,$ $\gamma \in \Gamma,\gamma' \in \Gamma ^{\mathbb{C}}.$ 

Moreover, $\frac{d^{2}}{du^{2}}\Psi \left( p,e^{i\lambda u}\right) \geq 0$
for all $\lambda \in Lie\left( \Gamma \right) $ with equality if and only if
the vector field ${X}_{\lambda }\left(
e^{i\lambda u}p\right) =0.$
\end{proposition}

Let us apply the previous results in our set-up. 
We introduce some classical functionals on the space of K\"ahler potentials. The energy functionals
$I$, $J$, introduced by T. Aubin in \cite{Aub} (see also \cite{Ti2}), are defined for each pair 
$(\omega,\omega_{\phi}:=\omega+\ddbar\phi)$ by
\begin{equation*}
I(\omega,\omega_{\phi})=\frac{1}{V}\int_M \sqrt{-1}\partial\phi\wedge\bar{\partial}\phi\wedge\sum_{i=0}^{n-1}\omega^{i}\wedge\omega_{\phi}^{n-1-i}
=\frac{1}{V}\int_M\phi(\omega^n-\omega_\phi^n)\label{Ieq},
\end{equation*}
\begin{equation*}
J(\omega,\omega_\phi)=\frac{1}{V(n+1)}\int_M\sqrt{-1}\partial{\phi}\wedge\bar{\partial}\phi\wedge\sum_{i=0}^{n-1}(n-i)\omega^{i}\wedge\omega_{\phi}^{n-1-i},\label{Jeq}
\end{equation*}
where we have skipped again the normalization of the volume form by the factor $n!$ for the sake of clearness. Note that one has the relationship
$$J(\omega,\omega_\phi)=\int_0^1 \frac{I(\omega, \omega+s\ddbar \phi)}{s}ds.$$ 
\par It is well known that $I,J$ and $I-J$ are all nonnegative and equivalent. One may also define these functionals via a variational formula and they are very natural from this point of view. We refer to the recent work of \cite{BBGZ} where this idea is exploited in details.
If $\omega_{\phi_t}$ is a smooth path in the K\"ahler cone, a direct computation gives $$\frac{d }{dt}J(\omega,\omega_{\phi_t}) = \frac{1}{V}\int_M \dot{\phi}_t (\omega^n - \omega_{\phi_t}^n).$$
We obtain
\begin{proposition}
 The integral of the moment map associated to $\overline{\mu}: \mathcal{M}\rightarrow Lie(\mathcal{G})^*$ is given by  the functional 
$$F^0_\Omega(\omega,\omega_\phi)=J(\omega,\omega_\phi)+\frac{1}{V}\int_M \phi(\Omega-\omega^n).$$
\end{proposition}
In particular, this functional is decreasing along the $\Omega$-K\"ahler flow. This can be checked directly since, along the flow (\ref{MA2}), one has 
$$\frac{d }{dt} F^0_\Omega(\omega,\omega_{\phi_t})=  \int_M \dot{\phi}_t (\Omega - \omega_{\phi_t}^n) = -\int_M \frac{ \left(\omega_{\phi_t}^n- \Omega\right)^2}{\omega_{\phi_t}^n}\leq 0.$$
On another hand, for the second derivative along the  $\Omega$-K\"ahler flow, one gets that
\begin{eqnarray}
\frac{d^2 }{dt^2} F^0_\Omega(\omega,\omega_{\phi_t})&=&\int_M \frac{\dot{\phi}_t\ddot{\phi}_t(-2+\dot{\phi}_t)}{(1-\dot{\phi}_t)^2}\Omega \label{secondd}
\end{eqnarray}
by using the fact that
$\frac{d \omega_{\phi_t}^n}{dt}=\frac{\ddot{\phi}_t}{(1-\dot{\phi}_t)^2}\Omega$. 
The term $-2+\dot{\phi}_t$ is always negative from Equation (\ref{MA2}), since the term $\frac{\Omega}{\omega_{\phi_t}^n}$ is positive. We can apply the maximum principle to check that the derivative of $(\dot{\phi}_t)^2$ is negative. So the right hand side of (\ref{secondd}) is actually positive
and the functional $t\mapsto F^0_\Omega(\omega,\omega_{\phi_t})$ is convex along the $\Omega$-K\"ahler flow.
Finally, it is clear that this functional satisfies the cocyclicity property 
$$F^0_\Omega(\omega,\omega_{\phi_1})=F^0_\Omega(\omega,\omega_{\phi_2})+F^0_\Omega(\omega_{\phi_2},\omega_{\phi_1}).$$
\subsection{The degenerate cases}\label{degen}
One can ask if the main results of this paper hold at least partially when one considers non ample classes or degenerate volume forms. Let us explain which arguments used in the previous sections can be extended with no major difficulty to show the existence of the balancing flow and its convergence when $k\rightarrow +\infty$ for non-smooth volume forms. \par First of all, a careful reading of the proof of Proposition \ref{tian-b} shows that the asymptotic expansion of the Bergman function holds when one considers a positive volume form $\Omega$ that can be written as $$\Omega=f_\Omega\omega^n$$ with $f_\Omega>0$ on $M$ and $f_\Omega\in L^1_\omega(M,\mathbb{R})$ and $\omega$ a smooth K\"ahler form.  Then, the asymptotic result for the operator $Q_k$  (Theorem \ref{quantlaplacien}) is valid when applied to the space of functions $f\in L^p_\omega(M,\mathbb{R})$ with $p>1$. This comes from the techniques of \cite{L-M} that can be extended from $L^2$ to $L^p$ topology, $p>1$. To be more precise, the regularity of the function $f$ is only needed in \cite[Equation (27)]{L-M} and the Cauchy-Schwartz inequality in \cite[Equation (28)]{L-M} can still be applied in the $L^p$ spaces. This implies that we get a more general version of Theorem \ref{identify} for $\Omega\in L^p(M)$ positive volume form ($p>1$) but with a weaker underlying convergence (the error terms are only controlled in $L^p$ norms instead of $C^ \infty$ norm for the sequence $\omega_k(t)$ and $\frac{\partial \omega_k(t)}{\partial t}$). Finally, when one considers non smooth forms $\Omega=f_\Omega\omega^n$ with $f_\Omega\in L^p_\omega(M), f_\Omega>0$ with $p>1$, the limit of the balancing flows is still the $\Omega$-K\"ahler flow (\ref{MA}). Note that we don't expect the potential of the involved metric in (\ref{MA}) to be smooth and we shall speak instead of weak $\Omega$-K\"ahler flow. We also remark that the balancing flow  will converge towards a balanced metric again. Actually a notion of balanced metric for $L^p$ volume forms (and even more general) has been studied in details in the recent work \cite[Section 7]{BBGZ}. Furthermore the technical results of Section \ref{proof} still hold. Thus, we are able to derive an analogue of  Theorem \ref{mainthm} as soon as we have an  existence and convergence result of the weak $\Omega$-K\"ahler flow in the infinite dimensional setup. 
\par Moreover we expect that this method can be generalized with no major difficulty to $L^p$ semi-positive forms, i.e. $f_\Omega\geq 0$. This will be studied in a forthcoming paper using the techniques developed by S. Kolodziej in his generalization of the Calabi problem \cite{Ko1}.
\medskip 
\par Finally, if one considers $L$ to be a big line bundle and if one assumes the existence of $\omega_0\in c_1(L)$ a semi-positive smooth closed real (1,1)-form, the results of R. Berman \cite[Theorems 1.1, 1.2]{Ber1} can be applied to derive a version of Proposition \ref{tian-b}. In that case, the equilibrium metric corresponds to $\omega_0$ and one has in the sense of measures the asymptotics expansion given by Equation (\ref{eqn12}). In this setup, we believe that what is happening on the space of K\"ahler potentials for the $\Omega$-K\"ahler flow is related to the very general version of the Calabi problem studied by the viscosity method {\it \`a la} P.L. Lions in \cite{EGZ,EGZ0}. In particular
we expect the following conjecture to hold.
\begin{conjecture}
 Assume that $L$ is a big line bundle, $\Omega> 0$ a volume form with $L^p$ density, $p>1$ such that $\int_M \omega_0^n = \int_M \Omega$. Then the $\Omega$-K\"ahler flow exists in a weak sense and converges towards a solution of the Calabi problem $$(\omega+\ddbar \phi_\infty)^n =\Omega$$
with $\phi_\infty\in C^0(M)$, $\phi_\infty\in C^{1,1}(K)$ for any compact subset $K\subset M\setminus \mathbb{B}_+(L)$ where $\mathbb{B}_+(L)$  is the augmented base locus of $L$, which is an analytic subvariety of $M$.
\end{conjecture}

\bibliography{balflow2b.bib}

\bigskip 
\noindent
\texttt{H. -D. Cao}\\
{\sl Department of Mathematics, Lehigh University, \\ Bethlehem, PA 18015 USA\\
Email address: \url{huc2@lehigh.edu}}
\bigskip

\noindent
\texttt{J. Keller}\\
{\sl Centre de Mathématiques et Informatique, Université de Provence, \\
Technopôle Ch\^ateau-Gombert, 39, rue F. Joliot Curie,  \\
13453 Marseille Cedex 13, FRANCE\\
Email address: \url{jkeller@cmi.univ-mrs.fr}}

\end{document}